\documentclass[a4paper, 11pt]{article}
\title{Cohomology and products of real weight filtrations}
\author{Thierry Limoges and Fabien Priziac}
\date{}

\usepackage{amsfonts}
\usepackage{amsmath}
\usepackage{amsthm}
\usepackage{amssymb}
\usepackage[all]{xy}
\usepackage{graphicx}
\usepackage{euscript}
\usepackage{amsfonts}
\usepackage{amsmath,amscd, amsthm, amssymb}
\usepackage[english]{babel}
\usepackage[latin1]{inputenc}
\usepackage{graphicx}
\usepackage{color}

\newtheorem{thm}{Theorem}[section]
\newtheorem{cor}[thm]{Corollary}
\newtheorem{lem}[thm]{Lemma}
\newtheorem{prop}[thm]{Proposition}
\theoremstyle{definition}

\newtheorem{definition}[thm]{Definition}

\theoremstyle{remark}
\newtheorem{rem}[thm]{Remark}

\numberwithin{equation}{section}

\newcommand{\supp}{\mathrm{supp}}
\newcommand{\im}{\mathrm{im}}

\newcommand{\RR}{\mathbb{R}}

\newcommand{\ZZ}{\mathbb{Z}}
\newcommand{\NN}{\mathbb{N}}

\newcommand{\dis}{\displaystyle}

\newcommand{\lee}{\leqslant}
\newcommand{\gee}{\geqslant}

\newcommand{\WC}{\mathcal W C^*}

\newcommand{\Cc}{\mathcal C}
\newcommand{\Gc}{\mathcal G}

\newcommand{\Wc}{\mathcal{W}}
\newcommand{\Nc}{\mathcal N}
\newcommand{\Kc}{\mathcal K}

\newcommand{\Hoc}{H o \, \mathcal C}

\newcommand{\Schc}{\mathrm{\bf{Sch}}_c(\RR)}
\newcommand{\Reg}{\mathrm{\bf{Reg}}_{comp}(\RR)}
\newcommand{\V}{\mathrm{\bf{V}}(\RR)}
\newcommand{\lra}{\longrightarrow}
\newcommand{\lmt}{\longmapsto}
\newcommand{\sg}{{\bf s}}
\newcommand{\Dec}{\mathrm{Dec}}

\newcommand{\otz}{\otimes_{\ZZ_2}}
\newcommand{\Wcc}{\mathcal W C_*}

\newcommand{\Xbar}{\overline X}

\renewcommand{\Xbar}{\overline X}

\newcommand{\hra}{\hookrightarrow}
\newcommand{\AS}{\mathcal{AS}}

\newcommand{\lk}{\mathrm{lk}}
\newcommand{\lla}{\longleftarrow}
\newcommand{\ot}{\otimes}

\begin{document}

\maketitle

\begin{abstract}
We associate to each algebraic variety defined over $\RR$ a filtered cochain complex, which computes the cohomology with compact supports and $\ZZ_2$-coefficients of the set of its real points. This filtered complex is additive for closed inclusions and acyclic for resolution of singularities, and is unique up to filtered quasi-isomorphism. It is represented by the dual filtration of the geometric filtration on semialgebraic chains with closed supports defined by McCrory and Parusi\'nski, and leads to a spectral sequence which computes the weight filtration on cohomology with compact supports. This spectral sequence is a natural invariant which contains the additive virtual Betti numbers.

We then show that the cross product of two varieties allows us to compare the product of their respective weight complexes and spectral sequences with those of their product, and prove that the cup and cap products in cohomology and homology are filtered with respect to the real weight filtrations.
\end{abstract}

\section{Introduction}

\selectlanguage{english}

In \cite{Del}, Deligne showed the existence of a weight filtration on the rational cohomology with compact supports of complex algebraic varieties, using mixed Hodge structures. In the real case, where there is no such structure, using the work of Guillén and Navarro-Aznar on cubical hyperresolutions (\cite{GNAPP} and \cite{GNA}), Totaro proposed in \cite{Tot} an analog of the weight filtration on the cohomology with compact supports and $\ZZ_2$-coefficients of the set of real points of real algebraic varieties, and in \cite{MP}, McCrory and Parusi\'nski developped a homological analog on the Borel-Moore homology with $\ZZ_2$-coefficients (in a subsequent paper, McCrory and Parusi\'nski also defined and studied a weight filtration on classical homology, with compact supports and coefficients in $\mathbb{Z}_2$ : see \cite{MP2}).

In \cite{MP}, McCrory and Parusi\'nski showed furthermore that the spectral sequence inducing the weight filtration (unlike the complex case, this spectral sequence does not degenerate at level two in general) is itself a natural invariant of real algebraic varieties : there is a functor that assigns to each real algebraic variety a filtered chain complex, the weight complex, which is functorial for proper regular morphisms, additive for closed inclusions, acyclic for resolution of singularities, and unique up to filtered quasi-isomorphism, inducing the weight spectral sequence and the weight filtration on Borel-Moore homology. We can also extract the virtual Betti numbers (\cite{MPVB}) from the weight spectral sequence. Moreover, the weight complex can be realized at the chain level by a filtration defined on semialgebraic chains with closed supports, using resolution of singularities. This geometric filtration is functorial with respect to semialgebraic maps with $\mathcal{AS}$ graph.
\\

In this paper, we first achieve the cohomological counterpart of McCrory and Parusi\'nski's work. Using the extension criterion of Guill\'en and Navarro-Aznar (\cite{GNA}), we associate to any real algebraic variety a filtered cochain complex, that we call the cohomological weight complex, which induces Totaro's weight filtration on the cohomology with compact supports and $\mathbb{Z}_2$-coefficients of the set of real points of real algebraic varieties  as well as the associated cohomological weight spectral sequence (theorem \ref{extension}, proposition \ref{propweight}, corollary \ref{fcoho} and proposition \ref{hcspec}).   

In section \ref{dgf}, we construct a filtration that realizes at the cochain level the cohomological weight complex (theorem \ref{coincid}) : it is a dualization of the geometric filtration of \cite{MP}. This allows us to show in particular that the cohomological weight spectral sequence and filtration are dual to the homological ones (corollary \ref{wfaredual}).
\\

The second part of this paper is devoted to the question of the compatibility of the real weight filtrations with products. First, if $X$ and $Y$ are two real algebraic varieties, we define the product of two respective semialgebraic chains of $X$ and $Y$ in a natural way (definition \ref{defprod}). We then look at its compatibility with the geometric filtration (proposition \ref{prodchain}). Finally, we show that there is a filtered quasi-isomorphism between the tensor product of the geometric filtrations of $X$ and $Y$ and the geometric filtration of the cross product $X \times Y$ (theorem \ref{prodhom}). In particular, the weight complex of the product is isomorphic to the product of the weight complexes and the K\"unneth isomorphism is filtered with respect to the weight filtration. The induced relations on weight spectral sequences can also be used to prove the multiplicativity of the virtual Poincar\'e polynomial. These results have then their cohomological counterparts (paragraph \ref{pcwcsect}).

Finally, for $Y = X$, we define a cup product and a cap product on the dual geometric and geometric filtrations of $X$ considered in the category of filtered complexes localized with respect to filtered quasi-isomorphisms (paragraphs \ref{cupsect} and \ref{capsect}), proving in particular that the usual cup and cap products in cohomology and homology are filtered with respect to the weight filtrations of $X$. In the last subsection \ref{wfpdmsect}, we study the cap product with the fundamental class of a compact variety at the weight spectral sequences level, which brings us some obstructions for a compact singular real algebraic variety to satisfy Poincar\'e duality.
\\

{\bf Acknowledgements.} The authors wish to thank G. Fichou, C. McCrory and A. Parusi\'nski for useful discussions and comments. 

\section{Framework} \label{framesect}

In this section, we set the context in which we work in this paper. We first fix precisely the source categories (categories of schemes over $\RR$) and target categories (categories of filtered chain and cochain complexes) of the functors we are going to deal with. We then describe the geometry of real algebraic varieties we study.

\subsection{Filtered cochain complexes} \label{cat}

In this paper, we will work with cochain complexes equipped with a decreasing filtration. We will use the following notations :
\begin{itemize}
	\item 
$\mathfrak{C}$ will denote the category of bounded cochain complexes of $\ZZ_2$-vector spaces with bounded decreasing filtration.\\

To each filtered complex $(K^*,F^{\bullet})$ of $\mathfrak{C}$ is associated a second quadrant spectral sequence $E$ with
$$E_0^{p,q}= \frac{F^{p} K^{p+q}}{F^{p+1} K^{p+q}} \mbox{ and } E_1^{p,q}= H^{p+q}\left(\frac{F^{p} K^*}{F^{p+1} K^*}\right),$$
the differential $d$ of the spectral sequence being induced by the coboundary operator of $K^*$.

The filtration $F^{\bullet}$ on $K^*$ induces naturally a filtration on the cohomology of $K^*$ by setting
$$F^{p} H^{q}(K) := \im \left[ H^q(F^pK^*) \lra H^q(K^*)\right],$$
and the spectral sequence $E$ converges to the cohomology $H^*(K^*)$ of $K^*$, that is
$$E_{\infty}^{p,q}= \frac{F^{p} H^{p+q}(K^*)}{F^{p+1} H^{p+q}(K^*)}.$$ 

A morphism of filtered complexes which induces an isomorphism on $E_1$ (and therefore on all $E_r$ from $r \gee 1$) will be called a \emph{quasi-isomorphism of $\mathfrak{C}$}, or simply a \emph{filtered quasi-isomorphism}.

\item
$H o \, \mathfrak{C}$ denotes the localization of $\mathfrak{C}$ with respect to quasi-isomorphisms of $\mathfrak{C}$ (we keep the notation of \cite{GNA}, 1.5.1).

\item
$\mathfrak{D}$ will denote the category of cochain complexes of $\ZZ_2$-vector spaces, and a morphism of cochain complexes which induces an isomorphism on their cohomology will be called a \emph{quasi-isomorphism of $\mathfrak{D}$}, or simply a \emph{quasi-isomorphism}.
\item
$H o \, \mathfrak{D}$ denotes the localisation of $\mathfrak{D}$ with respect to quasi-isomorphisms of $\mathfrak{D}$.
\end{itemize}

\begin{rem}
Let $\phi \ : \ \mathfrak{C} \lra \mathfrak{D}$ be the forgetful functor forgetting the filtration. It can be localized into a functor $\ H o \, \mathfrak{C} \lra H o \, \mathfrak{D}$, which we denote again by $\phi$, because a filtered quasi-isomorphism is in particular a quasi-isomorphism.
\end{rem}

\subsection{Real algebraic varieties}

We are  interested in the study of the geometry of the set of real points of real algebraic varieties. In this paper, a real algebraic variety will be a reduced scheme of finite type defined over $\RR$. We denote by :
\begin{itemize}
\item
$\Schc$ the category of real algebraic varieties and proper regular morphisms.
\item
$\Reg$ the full subcategory of $\Schc$ whose objects are compact nonsingular varieties, that is proper regular schemes.
\item
$\V$ the full subcategory of $\Schc$ whose objects are nonsingular projective varieties, that is regular projective schemes.
\end{itemize}

For $X$ a real algebraic variety of $\Schc$, we denote by $X_{\RR}$ the set of its real points. Equipped with its sheaf of regular functions, the set $X_{\RR}$ is a real algebraic variety in the sense of \cite{BCR}, which can be locally embedded in an affine space $\RR^n$. We equip it with the strong topology of $\RR^n$, and then $X_{\RR}$ is a Haussdorff space, locally compact.

\subsection{Semialgebraic chain and cochain complexes} \label{semich}

Let $X$ be a real algebraic variety. We will consider complexes of semialgebraic chains defined using semialgebraic subsets of $X_{\mathbb{R}}$. In this paper, we will always work with $\ZZ_2$-coefficients, so that real algebraic varieties and arc-symmetric sets (\cite{BCR}, \cite{KP}) always have a ($\mathbb{Z}_2$-)orientation and a fundamental class (recall that real algebraic varieties may not be $\mathbb{Z}$-oriented, as $\mathbb P^2(\RR)$).

We will consider the two following dual complexes : 
\begin{itemize}
\item 
the chain complex $(C_*(X),\partial_*)$ of semialgebraic chains of $X_{\mathbb{R}}$ with closed supports, whose homology is the Borel-Moore homology $H_*(X) := H_*^{BM}(X_{\mathbb{R}}, \mathbb{Z}_2)$ of $X_{\mathbb{R}}$ with coefficients in $\mathbb{Z}_2$ (see Appendix, paragraph 6 of \cite{MP}),

\item
the cochain complex $(C^*(X),\delta^*)$ which will be by definition the dual of $(C_*(X),\partial_*)$ and whose cohomology $(H_*(X))^{\vee}$ is, by proposition \ref{homdual} below, isomorphic to the cohomology with compact supports $H^*_c(X_{\mathbb{R}}, \mathbb{Z}_2)$ of $X_{\mathbb{R}}$ with coefficients in $\mathbb{Z}_2$.
\end{itemize}

\begin{prop} \label{homdual}
With the notations above, the Borel-Moore homology $H_*(X) := H_*^{BM}(X_{\mathbb{R}}, \mathbb{Z}_2)$ and the cohomology with compact supports $H^*(X) := H^*_c(X_{\mathbb{R}}, \mathbb{Z}_2)$ of $X(\mathbb{R})$ are dual :
$$H^*(X) = \left( H_*(X) \right) ^{\vee}$$
\end{prop}

\begin{proof} Borel-Moore homology and cohomology with compact supports of $X_{\mathbb{R}}$ can be defined respectively as the relative homology and cohomology of the pair $(\Xbar_{\mathbb{R}} , \Xbar_{\mathbb{R}} \setminus X_{\mathbb{R}})$, where $X \hookrightarrow \overline X$ is an open compactification of $X$.

The closed inclusion $\Xbar_{\mathbb{R}} \setminus X_{\mathbb{R}} \hra \Xbar_{\mathbb{R}}$ induces the following long exact sequences of homology and cohomology of the pair $(\Xbar_{\mathbb{R}} , \Xbar_{\mathbb{R}} \setminus X_{\mathbb{R}})$ :
$$\cdots  \lra H^{BM}_n(\Xbar_{\mathbb{R}} \setminus X_{\mathbb{R}} ) \lra  H^{BM}_n(\Xbar_{\mathbb{R}}) \lra H^{BM}_n(X_{\mathbb{R}})  \lra H^{BM}_{n-1}(\Xbar_{\mathbb{R}}  \setminus X_{\mathbb{R}} ) \lra \cdots$$
$$\cdots  \lla H^n_c(\Xbar_{\mathbb{R}} \setminus X_{\mathbb{R}} ) \lla  H^n_c(\Xbar_{\mathbb{R}}) \lla H^n_c(X_{\mathbb{R}})  \lla H^{n-1}_c(\Xbar_{\mathbb{R}}  \setminus X_{\mathbb{R}} ) \lla \cdots$$
The dual of the first sequence is then isomorphic to the second one by the five lemma, because the sets $\overline X_{\mathbb{R}}$ and $\overline X_{\mathbb{R}} \setminus X_{\mathbb{R}}$ are compact and consequently their Borel-Moore homology and cohomology with compact supports are respectively isomorphic to their singular homology and cohomology which are dual to each other.
\end{proof}

\begin{rem} The cohomology with compact supports is normally computed from the complex of cochains with compact supports. However, this complex does not have good additivity properties, in contrast with the complex $C^*$.
\end{rem}

\section{Cohomological weight complex} \label{cwf}

We prove the existence and uniqueness of the cohomological weight complex in a way similar to the method of \cite{MP}, using Theorem 2.2.2 of Guillén and Navarro-Aznar in \cite{GNA}. 

The Theorem 2.2.2 of \cite{GNA} is a criterion of extension to all (possibly singular or non-compact) varieties for functors defined on nonsingular projective varieties. In this paper we consider varieties over $\RR$ and the target category will be $H o \, \mathfrak{C}$. The initial functor will be the functor which assigns to a real nonsingular projective variety its complex of semialgebraic cochains equipped with the canonical filtration defined below (definition \ref{canco}). The verification of the criterion will follow from a short exact sequence for the cohomology of a blowing-up (remark \ref{seqcobl}), showing the existence of a unique extension to all real algebraic varieties and proper regular morphisms (which we call the cohomological weight complex), satisfying conditions of additivity for closed inclusions and acyclicity for generalized blow-ups (paragraph \ref{conscow}).

In paragraph \ref{cow}, we show that the spectral sequence induced by the cohomological weight complex (well-defined only from level one) converges to the cohomogical weight filtration on the cohomology with compact supports and, in paragraph \ref{cowvb}, that one can recover, as in \cite{MP}, section 1, the virtual Betti numbers (\cite{MPVB}) from its level one terms. The virtual Betti numbers are the unique additive invariants of real algebraic varieties coinciding with the usual Betti numbers on compact nonsingular varieties.

Finally, in paragraph \ref{wcub}, we give to the cohomological weight spectral sequence of a compact real algebraic variety the following viewpoint : it can be regarded as the spectral sequence naturally induced from a cubical hyperresolution of the variety.

\begin{definition} \label{defsg}
Keeping the notations from \cite{MP} and \cite{GNA}, for $n \gee 0$, let $\square^+_n$ be the set of subsets of $\{0,1,\dots,n\}$, partially ordered by inclusion. A \emph{cubical diagram} of type $\square^+_n$ in a category $\mathcal X$ is by definition a contravariant functor from $\square^+_n$ to $\mathcal X$. If $\Kc$ is a cubical diagram of type $\square^+_n$ in $\mathfrak{C}$, let $K^{*,S}$ be the complex labelled by the subset $S\subset \{0,1,\dots n\}$ and $|S|$ denote the number of elements of $S$. The {\it simple complex} $\sg \Kc$ is defined by
$$\sg \Kc^k :=  \bigoplus_{i+|S|=k} \Kc^{i,S}$$
with differentials $\delta \ : \ \sg \Kc^k \lra \sg \Kc^{k+1}$
defined as follows. 

For each $S$, let $\delta' \ : \ K^{i,S} \lra K^{i+1,S}$ be the differential of $K^{*,S}$.
If $T\subset S$ and $|T| = |S|-1$, let $\delta_{S,T} \ : \ K^{*,S} \lra K^{*,T}$ be the cochain map corresponding
to the inclusion of $T$ in $S$. If $a \in K^{i,S}$, let $\delta''(a) := \sum \delta_{S,T}(a)$ where the sum is taken over all $T\subset S$ such that $|T| = |S|-1$, and $\delta(a) := \delta'(a) + \delta''(a)$.

The induced filtration on $\sg \Kc$ is given by
$F^p \sg \Kc := \sg F^p \Kc$.
\end{definition}

\begin{rem} We distinguish the simple complex of a cubical diagram of type $\square^+_n$ of (filtered) cochain or chain complexes, as defined in the definition \ref{defsg} above, and the total complex associated to a double complex. 
\end{rem}

\begin{definition} \label{canco}
Let $(K^*, \delta)$ be a cochain complex. We define the canonical filtration $F_{can}^{\bullet}$ by $$F_{can}^p K^q := \left\{ \begin{array}{ll}
K^q &  \mathrm{\ if\ } q < -p\\
\ker \delta_q &  \mathrm{\ if\ } q = -p\\
0 &  \mathrm{\ if\ } q > -p
\end{array} \right.$$
\end{definition}

The filtered complex $F_{can}^{\bullet} K^*$ defines a second quadrant spectral sequence that converges to the cohomology of $K^*$ at level one :
\begin{equation} \label{speccan} 
E_{\infty}^{p,q} = E_1^{p,q} = \left\{ \begin{array}{ll}
\dis \frac{\ker \delta_{-p}}{\im \delta_{-p-1}} = H^{p+q}(K^*) &  \mathrm{\ if\ } p+q = -p,\\
0 &  \mathrm{\ otherwise.}
\end{array} \right.
\end{equation}

\subsection{The construction of the cohomological weight complex} \label{conscow}

We define a functor $\WC$ : $\Schc \longrightarrow H o \, \mathfrak{C}$ such that, for $X$ an objet of $\Schc$, the homology of the complex $\phi ( \WC (X) )$ is $H^*(X)$ (recall that $\phi$ denotes the forgetful functor). The spectral sequence $E_r$, $r = 1,2, \dots$ associated to $\WC (X)$, converges to $H^*(X)$. In particular, it induces a filtration on the cohomology with compact supports of $X_{\mathbb{R}}$. 

Theorem {\bf (2.2.2)} of \cite{GNA} allows us to prove the existence and uniqueness of the functor $\WC$ with properties of extension, additivity and acyclicity. We keep the notations from \cite{GNA} and \cite{MP}. 

\begin{thm}\label{extension}
The contravariant functor 
$$F_{can}C^* \ : \ \V \lra H o \, \mathfrak{C}$$
which assigns to a nonsingular projective variety $M$ the semialgebraic cochain complex with closed supports $C^*(M)$ equipped with the canonical filtration extends to a contravariant functor
$$\WC \ : \ \Schc \lra H o \, \mathfrak{C}$$
satisfying :\\

(Ac) For an acyclic square $\begin{array}{ccc}
\widetilde Y & \stackrel{j}{\hookrightarrow} & \widetilde X\\
\downarrow_{\pi} & & \downarrow_{\pi}\\
Y & \stackrel{i}{\hookrightarrow} & X
\end{array}$ the simple filtered complex of the diagram
$$\begin{array}{ccc}
\WC(\widetilde Y) & \stackrel{j^*}{\longleftarrow} & \WC(\widetilde X)\\
\uparrow_{\pi^*} & & \uparrow_{\pi^*}\\
\WC(Y) & \stackrel{i^*}{\longleftarrow} & \WC(X)
\end{array}$$
is acyclic.\\

(Ad) For a closed inclusion $Y \stackrel{i}{\hookrightarrow} X$, the simple filtered complex of the diagram
$$\WC(Y) \stackrel{i^*}{\longleftarrow} \WC(X)$$
is quasi-isomorphic in $\mathfrak{C}$ to $\WC(X \setminus Y)$.\\

Such a functor $\WC$ is unique up to a unique filtered quasi-isomorphism.
\end{thm}

\begin{rem} \label{seqcobl}
The proof uses ingredients analogous to the ones in homological weight complex existence and uniqueness' proof in \cite{MP}, in particular the fact that, for an elementary acyclic square
\begin{equation} \label{elemacysq}
\begin{array}{ccc}
\widetilde Y & \stackrel{j}{\hookrightarrow} & \widetilde X\\
\downarrow_{\pi} & & \downarrow_{\pi}\\
Y & \stackrel{i}{\hookrightarrow} & X
\end{array}
\end{equation}
the sequences
$$0 \lra H^{q}(X) \lra H^{q}(\widetilde X) \oplus H^{q}(Y) \lra H^{q}(\widetilde Y) \lra 0$$
are exact for all $q \in \NN$ (this uses Poincar\'e duality : see the proof of Proposition 2.1 of \cite{MPVB}).
\end{rem}

\begin{proof}\emph{(of Theorem \ref{extension})} Since the functor $F_{can}C^* \ : \ \V \lra H o \, \mathfrak{C}$ can be factorized through~$\mathfrak{C}$, it is $\Phi$-rectified. It remains to check the hypotheses (F1) and (F2) of theorem {\bf (2.2.2)} of \cite{GNA}.\\
\\
(F1) The inclusions $X \stackrel{i_X}{\hookrightarrow} X \sqcup Y$ and $Y \stackrel{i_Y}{\hookrightarrow} X \sqcup Y$ glue into an isomorphism $C^*(X \sqcup Y) \stackrel{i_X^* \oplus i_Y^*}{\lra} C^*(X) \oplus C^*(Y)$. As a consequence, $F_{can}C^*(X \sqcup Y) \cong F_{can}C^*(X) \oplus F_{can}C^*(Y)$.\\
(F2) For an elementary acyclic square (\ref{elemacysq}) we check that the following diagram, denoted by $\Kc$,
$$\begin{array}{ccc}
F_{can}C^*(\widetilde Y) & \stackrel{j^*}{\longleftarrow} & F_{can}C^*(\widetilde X)\\
\uparrow_{\pi^*} & & \uparrow_{\pi^*}\\
F_{can}C^*(Y) & \stackrel{i^*}{\longleftarrow} & F_{can}C^*(X)
\end{array}$$
is acyclic. In other words, we check that the spectral sequence associated to its simple filtered diagram satisfies $E_1(\sg \mathcal{K}) = 0$.

Let $p \in \mathbb{Z}$. Similarly to \cite{MP}, proof of Theorem 1.1, the homology of the $p$-th column $(E_0^{p,*}(\sg \mathcal{K}),d^{p,*})$ of $E_0(\sg \mathcal{K})$ is given by 
$$H^{k}(E_0^{p,*}) = \left\lbrace \begin{array}{ll}
 0 & \mathrm{\ for \ } k = -p,\\
 \dis \ker \left[ H^{-p}(X) \lra H^{-p}(\widetilde X) \oplus H^{-p}(Y) \right] & \mathrm{\ for \ } k = -p+1,\\
 \dis \frac{\ker \left[ H^{-p}(\widetilde X) \oplus H^{-p}(Y) \lra H^{-p}(\widetilde Y) \right]}{\im \left[ H^{-p}(X) \lra H^{-p}(\widetilde X) \oplus H^{-p}(Y) \right]} & \mathrm{\ for \ } k = -p+2,\\
 \dis \frac{H^{-p}(\widetilde Y)}{\im \left[ H^{-p}(\widetilde X) \oplus H^{-p}(Y) \lra H^{-p}(\widetilde Y) \right]} & \mathrm{\ for \ } k = -p+3,\\
 0 & \mathrm{\ otherwise.}
\end{array} \right.$$
These spaces are all $0$ (see previous remark \ref{seqcobl}) and therefore $E_1( \sg \Kc ) = 0$.
\end{proof}

\begin{rem} \label{remapext}

\begin{itemize} 
	\item The extension theorem 2.2.2 of \cite{GNA} gives us also the fact that the functor $\mathcal{W} C^*$ is $\Phi$-rectified.

	\item Let $(E_r)_{r \gee 0}$ be the spectral sequence associated to the filtered complex $\WC(X)$ provided by theorem \ref{extension}. By definition of the category $H o \, \mathfrak{C}$, the terms $E_r$ for $r = 1,2, \dots$ are well-defined and do not depend on the construction of $\WC(X)$. On the other hand, $E_0$ depends on this construction, that is on the chosen cubical hyperresolution of $X$ (see paragraph \ref{wcub} below).
\end{itemize}

\end{rem}

\subsection{Cohomological weight filtration} \label{cow}

For $X$ a real algebraic variety, we call the filtered complex $\WC(X)$ the \emph{cohomological weight complex} of $X$. 

The cohomological weight complex computes the cohomology with compact supports of the set of real points of real algebraic varieties :

\begin{prop} \label{propweight}
For $X$ an object of $\Schc$, the homology of the complex $\phi (\WC (X))$ is $H^*(X)$.
\end{prop}

\begin{proof}
The functors $\phi \circ \WC$ and $C^*( \cdot )$ both satisfy the additivity and acyclicity properties of Theorem {\bf (2.2.2)} of \cite{GNA} (we have short exact sequences of additivity and acyclicity for the semialgebraic cochain complex, which are dual to the ones in \cite{MP}, proof of Proposition 1.5). Since they are furthermore equal on objects of $\V$, these functors $\Schc \lra H o \, \mathfrak{D}$ are isomorphic, thanks to the uniqueness provided by the extension theorem. In particular, the semialgebraic cochain complex with closed supports and the cohomological weight complex compute the same cohomology.
\end{proof}

\begin{cor} \label{fcoho}
By theorem \ref{extension} and previous property \ref{propweight}, we obtain a filtration on the cohomology with compact supports of the set of real points of real algebraic varieties :
$$H^k(X) = \Wc^{-k} H^k(X) \supset \Wc^{-k+1} H^k(X) \supset \cdots \supset \Wc^{0} H^k(X) \supset \Wc^{+1} H^k(X) = \left\lbrace 0 \right\rbrace,$$
called the \emph{cohomological weight filtration}.

We say that the cohomological weight filtration of a variety $X$ is \emph{pure} if for all $k \in \mathbb{Z}$ the space $\Wc^{-k+1} H^k(X)$ is 0.
\end{cor}

It will be shown in section \ref{grad} that the cohomological weight filtration is dual to the weight filtration on Borel-Moore homology of \cite{MP}. In particular, the cohomological weight filtration of a real algebraic variety $X$ is pure if and only if its homological weight filtration is pure.

\subsection{Cohomological weight spectral sequence and virtual Betti numbers} \label{cowvb}

Analogously to \cite{MP}, we recover the virtual Betti numbers from the \emph{cohomological weight spectral sequence}, which is by definition the spectral sequence $E_r$ associated to the cohomological weight complex $\WC(X)$ of $X$ (it is well-defined for $r \gee 1$ : see remark \ref{remapext}). We reindex it by setting $\widetilde E_r^{p,q} = E_{r-1}^{-q,p+2q}$. Notice that the column $(-p)$ of $E_1$ is sent to the row $p$ of $\widetilde E_2$, the line $p+q=-p$ is sent to the vertical line $p=0$ and the lines $p+q$ = constant are globally preserved.

\begin{lem} \label{dimfinie}
The vector spaces appearing in $E_1$ (or equivalently $\widetilde E_2$) have finite dimension.
\end{lem}

\begin{proof}
It will be shown in paragraph \ref{wcub} that, for compact varieties, the spectral sequence $\widetilde E_r$ is isomorphic to the spectral sequence $\widehat E_r$ induced by the double complex associated to a cubical hyperresolution (from $r \gee 2$). Since the latter is computed from cohomologies of real algebraic varieties, its terms are finite-dimensional. We next use the additivity property of the cohomological weight complex to prove the non-compact case.
\end{proof}

\begin{prop} \label{vbnc}
The $q$-th virtual Betti number can be read on the $q$-th row of $\widetilde E_2$ :
$$\beta_q(X) = \sum_{p=0}^{\dim X} (-1)^p \dim_{\ZZ_2} \widetilde E_2^{p,q}(X)$$
\end{prop}

\begin{proof}
As in \cite{MP}, paragraph 1.3, the right-hand side of the formula equals $\dis b_q(X) := \dim_{\ZZ_2} H^q(X)$ for $X$ compact nonsingular (in this case,  according to proposition \ref{compnonsing} below, $\WC(X)$ is filtered quasi-isomorphic to $C^*(X)$ equipped with the canonical filtration), and is additive for a closed inclusion $Y \stackrel{i}{\hookrightarrow} X$ : this gives us the result by the uniqueness with such properties of the virtual Betti numbers, see \cite{MPVB}.

\end{proof}

\begin{prop} \label{compnonsing}
For $X$ compact nonsingular, the cohomology of the complex $\dis \frac{\mathcal W^pC^*(X)}{\mathcal W^{p+1}C^*(X)}$ satisfies
$$\dis H^k \left( \frac{\mathcal W^pC^*(X)}{\mathcal W^{p+1}C^*(X)} \right) = \left\{ \begin{array}{ll}
H^{p}(X) &  \mathrm{\ if\ } k = -p,\\
0 &  \mathrm{\ otherwise.}
\end{array} \right.$$
In other words, the cohomological weight filtration of a compact nonsingular real algebraic variety is pure.
\end{prop}

\begin{proof} The proof is similar to the proof of Proposition 1.8 in \cite{MP}. Since the inclusion $\mathbf{V}(\mathbb{R}) \rightarrow \mathbf{Reg}_{comp}(\mathbb{R})$ has the extension property of \cite{GNA} (2.1.10), the functor $F_{can}^* C^* : \mathbf{V}(\mathbb{R}) \rightarrow H o \, \mathfrak{C}$ extends to a functor $\mathbf{Reg}_{comp}(\mathbb{R}) \rightarrow H o \, \mathfrak{C}$, unique up to filtered quasi-isomorphism with extension, additivity and acyclicity properties. Since $F_{can}^* C^* : \mathbf{Reg}_{comp}(\mathbb{R}) \rightarrow H o \, \mathfrak{C}$ and $\mathcal{W} C^* :  \mathbf{Reg}_{comp}(\mathbb{R}) \rightarrow H o \, \mathfrak{C}$ are both such extensions (see the remark \ref{seqcobl} and the proof of theorem \ref{extension} : $F_{can}C^*$ is acyclic on acyclic squares in $\mathbf{Reg}_{comp}(\mathbb{R})$), they are quasi-isomorphic in $\mathfrak{C}$ and we obtain the result by the equality (\ref{speccan}).
\end{proof}

\subsection{Cohomological weight complex and cubical hyperresolutions} \label{wcub}

This subsection gives the following viewpoint for the weight complex : the spectral sequence given by theorem \ref{extension} is quasi-isomorphic to the spectral sequence induced by the double complex associated to a cubical hyperresolution. For an introduction to cubical hyperresolutions of algebraic varieties, see \cite{PSMHS}, ch. 5. 
\\

Here, we keep the notations from \cite{GNA}. For any compact real algebraic variety $X$, there exists a \emph{cubical hyperresolution} of $X$, which is a special case of cubical diagram (see definition \ref{defsg}) denoted by
$$X_{\bullet} = \left[ X_{\bullet}^+ \lra X \right],$$
where $X_{\bullet}^+$ is called the \emph{augmented cubical diagram}. 

A cubical hyperresolution $X_{\bullet}$ of $X$ is composed of varieties $X_S$, $S \in \mathcal{P}[\![1,n]\!]$, associated to the vertices of a $n$-dimensional cube, with $X_{\emptyset} = X$ and $X_S$ compact nonsingular for $S \neq \emptyset$, and of morphisms $\pi_{S,T}$ : $X_S \lra X_T$ for $T \subset S$, such that $\pi_{R,T} = \pi_{R,S} \circ \pi_{S,T}$ if $T \subset S \subset R$. 
\\

Let $X_{\bullet}$ be a cubical hyperresolution of a compact real algebraic variety $X$. We associate to $X_{\bullet}$ a double complex $C^{i,j}$, defined by $X^{(i)} := \bigsqcup_{S \subset [\![ 1,n ]\!], |S|=i+1} X_S$ and
$$C^{i,j} := C^j(X^{(i)}) \cong \bigoplus_{|S|=i+1} C^j(X_S)$$
for $i, j \gee 0$, and equipped with the differentials $\delta'_i \ : \ C^j(X^{(i)}) \lra C^j(X^{(i+1)})$,
induced by the morphisms $X^{(i+1)} \lra X^{(i)}$ (that is by the morphisms $X_S \rightarrow X_T$ with $T \subset S$ and $|T| = |S| - 1$), and $\delta''_j \ : \ C^j(X^{(i)}) \lra C^{j+1}(X^{(i)})$ which are the coboundary operators of the $X^{(i)}$'s.
\\

The double complex $C^{i,j}$ leads to a filtered complex $(C^*, \hat F)$ : it is the associated total complex equipped with the \emph{naive} filtration coming from the cubical diagram structure. Precisely, we set
$$C^k := \bigoplus_{i+j=k} C^j(X^{(i)})$$ 
and
$$\hat F^p C^k := \bigoplus_{i \gee p} C^{k-i} (X^{(i)}),$$
the differential being $\delta := \delta' + \delta''$.

Similarly to \cite{MP}, Proposition 1.9, since the functor $\WC$ is  acyclic for cubical hyperresolutions (see \cite{GNA}, proof of Theorem 2.1.5), the cohomological weight spectral sequence of $X$ is isomorphic (from level one) to the spectral sequence associated to the cochain complex $C^*$ equipped with (the filtration induced by) the canonical filtration denoted by $F_{can}$. Now, the Deligne shift (\cite{Del2} Paragraph 1.3) of the naive filtration $\hat F$ on $C^*$ is the canonical filtration $F_{can}$~:

\begin{lem} \label{Delshift}
$$\Dec (\hat F)^{p-k} C^k = F_{can}^{p-k} C^k$$
\end{lem}

Consequently, the spectral sequence $\widehat E_r$ induced by the filtered complex $(C^*, \hat F)$ associated to the cubical hyperresolution $X_{\bullet}$ of $X$ is isomorphic to the reindexed cohomological weight spectral sequence of $X$ (from level two) :

\begin{prop} \label{hcspec}
For $r \gee 2$
$$\widetilde E_{r}^{a,b} = E_{r}^{a,b} (C^*, \hat F)$$
\end{prop}

\begin{rem} If $U$ is a non-compact real algebraic variety, take a compactification $X$ of $U$ and consider the complement $X \setminus U$ of $U$ in $X$. Then, using the additivity property of the weight complex and the proposition \ref{hcspec} above, we can compute the cohomological weight spectral sequence of $X$ from the spectral sequences induced by cubical hyperresolutions of $X$ and $X \setminus U$.
\end{rem}

\section{The dual geometric filtration} \label{dgf}

In \cite{MP}, McCrory and Parusi\'nski built a functor $\Gc_{\bullet} C_* \ : \ \Schc \lra \mathcal{C}$ (where $\mathcal{C}$ is the category of bounded filtered chain complexes, see \cite{MP}) representing the weight complex functor $\mathcal W C_*$ defined in $H o \, \mathcal{C}$ (up to filtered quasi-isomorphisms only). Dualizing the geometric filtration $\Gc_{\bullet}$, we obtain a functor representing the cohomological weight complex $\WC$ at the cochain level. Therefore, our cohomological weight complex functor
$$\WC \ : \ \Schc \lra H o \, \mathfrak{C}$$
can be factorized into a functor
$$\dis \Gc^{\bullet} C^* \ : \ \Schc \lra \mathfrak{C}$$
through the canonical localization $\mathfrak{C} \lra H o \, \mathfrak{C}$. 

Analogously we can dualize the Nash filtration $\Nc_{\bullet}$ (see section 3 of \cite{MP}) that extends the geometric filtration on the wider category $\chi_{\mathcal{AS}}$ of $\mathcal{AS}$-sets (see \cite{K88} and \cite{KP}), and then extend the functor $\Gc^{\bullet} C^*$ to the category $\chi_{\mathcal{AS}}$, showing in particular that the semialgebraic cochain complex equipped with the dualized geometric filtration is functorial with respect to semialgebraic morphisms with $\AS$ graph.

We remark also that the cohomological weight spectral sequence $E_r$ is dual to the homological weight spectral sequence for $r \gee 0$ and deduce that the cohomological weight filtration can be obtained by dualizing the homological weight filtration.

\subsection{Definition} \label{defdual}

Let $X$ be a real algebraic variety. We dualize the geometric filtration on the semialgebraic chain complex of (the set of real points of) $X$ in the following way. We set
$$\Gc^pC^q(X) := \left\{\varphi \in C^q(X) ~ | ~ \varphi \equiv 0 \mathrm{\ on \ } \Gc_{p-1} C_q(X) \right\}$$
i.e. $\Gc^pC^q(X)$ consists of the linear forms defined on $C_q(X)$ and vanishing on $\Gc_{p-1} C_q(X)$.

We get a decreasing filtration on $C^*(X)$ :
$$ C^k(X) = \Gc^{-k} C^k(X) \supset \Gc^{-k+1} C^k(X) \supset \cdots \supset \Gc^{0} C^k(X) \supset \Gc^{1} C^k(X) = 0,$$
that we call the {\it dual geometric filtration} or the {\it cohomological geometric filtration} of $X$. We show in proposition \ref{coincid} that the induced spectral sequence $E_r$ (well-defined for $r =0,1, \dots$ and functorial in $X$) coincides with the spectral sequence of the weight complex from $r \gee 1$.

The cohomological geometric filtration satisfies the properties of short exact sequences of additivity and acyclicity (lemma \ref{SEFD}), dual to the ones of Theorem {\bf (2.7)} and {\bf (3.6)} of \cite{MP}. These properties are stronger than the additivity and acyclicity properties $(Ad)$ and $(Ac)$ of $\WC$, which can be recovered by the snake lemma. 

\begin{rem} \label{grad}
We have $\dis \Gc^pC^q(X) \cong \left( \frac{C_q(X)}{\Gc_{p-1} C_q(X)} \right)^{\vee}$, since we can factorize a linear form on $C_q(X)$ which kernel contains $\Gc_{p-1} C_q(X)$ through $C_q(X) \rightarrow \frac{C_q(X)}{\Gc_{p-1} C_q(X)}$. Furthermore, if we consider the restriction of morphisms of $\Gc^pC^q(X)$ to $\Gc_{p} C_q(X)$, the quotients on chains and cochains are related by :
$$\frac{\Gc^{p} C^q(X)}{\Gc^{p+1} C^q(X)} = \left( \frac{\Gc_{p} C_q(X)}{\Gc_{p-1} C_q(X)} \right)^{\vee}.$$

Since the isomorphisms are compatible with the differentials of the complexes which are dual to one another, they induce a duality between the spectral sequence associated to the cochain complex $\Gc^{\bullet} C^*(X)$ and the spectral sequence associated to the chain complex $\Gc_{\bullet} C_*(X)$~:
$$E_r^{p,q} = (E^r_{p,q})^{\vee}$$
(from $r \gee 0$).
\\

Notice that the construction of the dual geometric filtration can be generalized to any filtered chain complex of $\mathcal{C}$, providing the same duality on the induced spectral sequences : if $(K_*, F_{\bullet}) \in \mathcal{C}$, we define the \emph{dual filtration} of $F_{\bullet}$ by
$$F_{\vee}^p K_{\vee}^{q} := \left\{\varphi \in (K_q)^{\vee} ~ | ~ \varphi \equiv 0 \mathrm{\ on \ } F_{p-1} K_q(X) \right\}$$
and then 
$$E_r^{p,q}(K_{\vee},F_{\vee}) = (E^r_{p,q}(K, F))^{\vee}$$
from $r \gee 0$.  
\end{rem}

\begin{lem}\label{SEFD}
For any closed inclusion $Y \stackrel{i}{\hookrightarrow} X$ and any $p,q \in \mathbb{Z}$, the sequence
$$0 \longrightarrow \Gc^pC^q(X \setminus Y) \stackrel{}{\longrightarrow} \Gc^pC^q(X)  \stackrel{}{\longrightarrow} \Gc^pC^q(Y) \longrightarrow 0$$
is exact.  

For an acyclic square $\begin{array}{ccc}
\widetilde Y & \stackrel{j}{\hookrightarrow} & \widetilde X\\
\downarrow_{\pi} & & \downarrow_{\pi}\\
Y & \stackrel{i}{\hookrightarrow} & X
\end{array}$ and any $p,q \in \mathbb{Z}$, the sequence
$$0 \longrightarrow \Gc^pC^q(X) \stackrel{}{\longrightarrow} \Gc^pC^q(\widetilde X) \oplus \Gc^pC^q(Y) \stackrel{}{\longrightarrow} \Gc^pC^q(\widetilde Y) \longrightarrow 0$$
is exact.
\end{lem}

\begin{proof}
We use the natural isomorphisms $\dis \Gc^pC^q(X) \cong \left( \frac{C_q(X)}{\Gc_{p-1} C_q(X)} \right)^{\vee}$ (remark \ref{grad}) and the short exact sequences of additivity and acyclicity of the homological geometric filtration (Theorem 2.7 of \cite{MP}).
\end{proof}

\subsection{Realization of the cohomological weight complex}

We now use the previous lemma \ref{SEFD} to show that the cohomological geometric filtration realizes the cohomological weight filtration :

\begin{prop} \label{coincid}
The dual geometric filtration $\Gc^{\bullet}C^* \ : \ \Schc \lra \Cc$ induces the functor $\WC \ : \ \Schc \lra \Hoc$.
\end{prop}

\begin{proof}
The functor $\Gc^{\bullet} C^*$, composed with the canonical localization $\mathfrak{C} \rightarrow H o \, \mathfrak{C}$ verifies the properties $(Ac)$ and $(Ad)$ of Theorem \ref{extension} by lemma \ref{SEFD}. If it also verifies the extension property, by uniqueness of the cohomological weight complex, the two functors $\Gc^{\bullet} C^*,~\WC : \Schc \lra H o \, \mathfrak{C}$ will be isomorphic.
\\

Let $X$ be a nonsingular projective real algebraic variety. We show that $\Gc^{\bullet} C^*(X)$ is filtered quasi-isomorphic to $F^{\bullet}_{can} C^*(X)$. According to \cite{MP} Theorem 2.8., the complexes $\Gc_{\bullet} C_*(X)$ and $F^{can}_{\bullet} C_*(X)$ are filtered quasi-isomorphic (through the inclusion morphism). By remark \ref{grad}, we deduce that, on the cohomological spectral sequences level,
$$E_1\left(\Gc^{\bullet} C^*\right) = \left(E^1(\Gc_{\bullet} C_*)\right)^{\vee} \cong \left(E^1(F^{can}_{\bullet} C_*(X))\right)^{\vee} = E_1\left(\left(F^{can} \right)_{\vee}^{\bullet} C^*(X)\right),$$
where the dual canonical filtration $\left(F^{can} \right)_{\vee}^{\bullet} C^*(X)$ is given by 
$$\left( F^{can} \right)_{\vee}^{p} C^q(X) = \left\{\varphi \in C^q(X) ~ | ~ \varphi \equiv 0 \mathrm{\ on \ } F^{can}_{p-1} C_q(X) \right\} = \left\{ \begin{array}{ll}
0 & \mathrm{\ if \ } q > - (p-1), \\
\im \ \delta_{q-1} & \mathrm{\ if \ } q =-(p-1),\\
C^q(X) & \mathrm{\ if \ } q < - (p-1),
\end{array} \right.$$
(notice that a linear form on $C_q(X)$ which vanishes on $\ker \partial_q$ can be factorised into a linear form on $C^{q-1}(X)$ through $\partial_q$ and then belongs to $\im \ \delta_{q-1}$).\\

We observe that there is an inclusion of the canonical filtration given by
$$F^p_{can} C^q(X) = \left\{ \begin{array}{ll}
0 & \mathrm{\ if \ }  q > - p, \\
\ker \delta_{q} & \mathrm{\ if \ } q = - p,\\
C^q(X) & \mathrm{\ if \ } q < - p,
\end{array} \right.$$
in the filtration $\left(F^{can} \right)_{\vee}^{\bullet} C^*(X)$, that induces a quasi-isomorphism of $\mathfrak{C}$. Indeed
$$E_1^{p,q}(\left(F^{can} \right)_{\vee}^{\bullet} C^*(X)) = (E^1_{p,q}(F^{can}_{\bullet} C_*(X)))^{\vee} = \left\{ \begin{array}{ll}
H^{p+q}(X) & \mathrm{\ if \ } p+q=-p,\\
0 & \mathrm{\ otherwise. \ }
\end{array} \right. = E_1^{p,q}\left(F^{\bullet}_{can} C^*(X)\right)$$ 
(see paragraph 1.1 of \cite{MP} and equality (\ref{speccan})).
\\

As a consequence, $E_1\left(\Gc^{\bullet} C^*\right) = E_1\left(\left(F^{can} \right)_{\vee}^{\bullet} C^*(X)\right) = E_1 \left(F^{\bullet}_{can} C^*(X)\right)$ and $\Gc^{\bullet} C^*(X)$ and $F^{\bullet}_{can} C^*(X)$ are isomorphic in $H o \, \mathfrak{C}$.
\end{proof}

\begin{cor} \label{wfaredual}
We have isomorphisms
$$\frac{\Wc^{p} H^q(X)}{\Wc^{p+1} H^q(X)} = \left( \frac{\Wc_{p} H_q(X)}{\Wc_{p-1} H_q(X)} \right)^{\vee},$$ 
and the weight filtrations on Borel-Moore homology in \cite{MP} and cohomology with compact supports in corollary \ref{fcoho} are related by
$$\Wc^p H^q(X) = \left\{\varphi \in H^q(X) ~|~ \varphi \equiv 0 \mathrm{\ on \ } \Wc_{p-1} H_q(X) \right\}$$
\end{cor}

\begin{proof}
We use the facts that the homological and cohomological geometric filtrations realize respectively the homological and cohomological weight spectral sequences, and that these spectral sequences are then dual to one another (remark \ref{grad}) : the first assertion is just the isomorphism
$$E_{\infty} = \left( E^{\infty} \right)^{\vee}.$$
We deduce that
$$\Wc^p H^q(X) = \left( \frac{H_q(X)}{\Wc_{p-1} H_q(X)} \right)^{\vee},$$
i.e.
$$\Wc^p H^q(X) = \left\{\varphi \in H^q(X)~|~\varphi \equiv 0 \mathrm{\ on \ } \Wc_{p-1} H_q(X) \right\}.$$

\end{proof}

In a similar way, we can define a filtration $\Nc^{\bullet}$ on the semialgebraic cochain complex $C^*(X)$ of an $\AS$-set, dual to the Nash filtration $\Nc_{\bullet}$ of \cite{MP}, section 3. This dual filtration defines a functor
$$\Nc^{\bullet} C^* \ : \ \chi_{\mathcal{AS}} \lra \mathfrak{C}$$
from the category of $\AS$-sets, which extends $\Gc^{\bullet}C^* \ : \ \Schc \lra \mathfrak{C}$. In particular we obtain :

\begin{prop} \label{ASdual}
The dual geometric filtration and its spectral sequence are functorial with respect to semialgebraic morphisms with $\AS$ graph.
\end{prop}

\section{Weight filtrations and products} \label{wfpsect}

In this section, we define the cross product of semialgebraic chains with closed supports (definition \ref{defprod}). We first check that it is well-defined and give its behaviour under the boundary operator (lemmas \ref{prodwelldef} and \ref{bordprod}) and with respect to the geometric filtration (proposition \ref{prodchain}). 

If $X$ and $Y$ are two real algebraic varieties, the product of chains then induces a well-defined morphism $u$ from the tensor product of the geometric filtrations of $X$ and $Y$ to the geometric filtration of the cross product $X \times Y$ (theorem \ref{prodhom}). Using the naturality property of the extension criterion of Guill\'en and Navarro-Aznar in \cite{GNA}, we show that it is a filtered quasi-isomorphism. Consequently, the tensor product of the weight complexes is isomorphic in $H o \, \mathcal{C}$ to the weight complex of the product. Looking at the induced relations between the weight spectral sequences terms, this implies in particular the multiplicativity of the virtual Poincar\'e polynomial (without the use of the weak factorization theorem) and the fact that the K\"unneth isomorphism is filtered with respect to the weight filtration. In paragraph \ref{pcwcsect}, dualizing the quasi-isomorphism $u$, we show the cohomological counterparts of these results. 

In subsection \ref{cupsect}, we use these cohomological counterparts to define a cup product on the dual geometric filtration of a real algebraic variety $X$ at the localized cochain level. It induces the usual cup product on the cohomology of its real points, showing that the latter is filtered with respect to the cohomological weight filtration. We then use this cup product in $H o \, \mathfrak{C}$ to also define a cap product on the cochain and chain level (subsection \ref{capsect}). Finally, in paragraph \ref{wfpdmsect}, we give obstructions for a compact real algebraic variety to satisfy Poincar\'e duality, relating to its weight filtrations.
 
\subsection{Product of semialgebraic chains} \label{pscsect}

Let $X$ and $Y$ be two real algebraic varieties. We define a product operation between the chains of $X$ and the chains of $Y$, checking in lemma \ref{prodwelldef} that this operation is well-defined.

\begin{definition} \label{defprod}
For any chains $c=[A] \in C_q(X)$ and $c'=[B] \in C_{q'}(Y)$, we define $$c \times c' := [A \times B] \in C_{q+q'}(X \times Y).$$
\end{definition}

\begin{lem} \label{prodwelldef}
Let $A$ and $A'$, respectively $B$ and $B'$, be two closed semialgebraic subsets of (the set of real points) of $X$, respectively $Y$, such that $[A] = [A']$ in $C_n(X)$ and $[B] = [B']$ in $C_m(Y)$ for some nonnegative integers $n$ and $m$.
Then
$$[A \times B] = [A' \times B']$$
in $C_{n+m}(X \times Y)$.
\end{lem}

%\begin{proof}
%We check that $[A \times B] + [A' \times B'] = 0$ in $C_{n+m}(X \times Y)$. By the definition of semialgebraic chains with closed supports in the Appendix of \cite{MP}, we have
%$$[A \times B] + [A' \times B'] = [cl_{X \times Y} \left( A \times B \div A' \times B'\right)].$$
%Since $A \times B \cup A' \times B'  \subset (A \cup A') \times (B \cup B')$ and $A \times B \cap A' \times B' = (A \cap A') \times (B \cap B')$, we have
%$$A \times B \div A' \times B' \subset (A \cup A') \times (B \cup B') \setminus (A \cap A') \times (B \cap B') = \left( (A \div A') \times (B \cup B') \right) \cup \left((A \cup A') \times (B \div B') \right).$$
%But $\dim (A \div A') < n$ and $\dim (B \div B') < m$, therefore
%$\dim \left( (A \div A') \times B \right) \cup \left(A \times (B \div B') \right) < n + m$ and
%$$\dim cl_{X \times Y} \left( A \times B \div A' \times B'\right) = \dim A \times B \div A' \times B' < n + m$$
%and $[A \times B] + [A' \times B'] = 0$.
%\end{proof}

We then verify that the product of chains is distributive over the sum :

\begin{lem}
If $c_1,~c_2$ are two chains of $X$ and $c'$ is a chain of $Y$,
$$(c_1 + c_2) \times c' = c_1 \times c' + c_2 \times c',$$
and if $c$ is a chain of $X$ and $c'_1,~c'_2$ are two chains of $Y$,
$$c \times (c'_1 + c'_2) = c \times c'_1 + c \times c'_2.$$
\end{lem}

%\begin{proof}

%We write $c_1 = [A_1],~c_2 = [A_2]$ and $c' = [B]$. We then have
%\begin{eqnarray*}
%c_1 \times c' + c_2 \times c' & = & [A_1 \times B] + [A_2 \times B] \\
%															& = & [cl_{X \times Y}\left( (A_1 \times B) \div (A_2 \times B) \right)] \\
%															& = & [cl_{X \times Y}\left( (A_1 \div A_2) \times B \right)] \\
%															& = & [cl_{X}\left( A_1 \div A_2 \right) \times B] \\
%															& = & [cl_{X}\left( A_1 \div A_2 \right)] \times [B] \\
%															& = & (c_1 + c_2) \times c'
%\end{eqnarray*}

%The equality $(c'_1 + c'_2) = c \times c'_1 + c \times c'_2$ comes from a symmetric computation.

%\end{proof}

The next lemma describes the behaviour of the semialgebraic boundary operator with respect to the product on semialgebraic chains we defined above.

\begin{lem} \label{bordprod} The boundary of the product of two chains $c \in C_q(X)$ and $c' \in C_{q'}(Y)$ verifies, in $C_{q+q'-1}(X \times Y)$,
$$\partial (c \times c') = \partial c \times c' + c \times \partial c'.$$ 
\end{lem}

\begin{proof} Let $A \subset X$ and $B \subset Y$ be closed semialgebraic sets representing respectively $c$ and $c'$. We show that
$$\partial(A \times B) = \partial A \times B \cup A \times \partial B.$$
%Then, by definition, the closed semialgebraic set $A \times B \subset X \times Y$ represents the chain $c \times c'$ and $\partial (c \times c') = [\partial(A \times B)]$ (see Appendix of \cite{MP}).

%and then $\partial (c \times c') = [\partial A \times B] + [A \times \partial B] = \partial c \times c' + c \times \partial c'$ (notice that $\partial A \times B \cap A \times \partial B = \partial A \times \partial B$ and $\dim \partial A \times \partial B \lee q + q' - 2$). 

First recall that, for a semialgebraic set $S$, $\partial S = \{x \in S~|~\chi(\lk(x, S)) \equiv 1 \mod 2\}$, where $\lk(x, S) := S(x,\epsilon) \cap S$ for $\epsilon$ small enough. In the lemma \ref{lkprod} below, we prove that, for a fixed point $(a,b) \in A \times B$, the link $\lk((a,b), A \times B)$ of $(a,b)$ in $A \times B$ is semialgebraically homeomorphic to the set
$$\{ \lambda(a, \beta) + (1-\lambda) (\alpha, b)~|~\lambda \in [0,1],~\beta \in \lk(b,B),~\alpha \in \lk(a,A) \}.$$

Applying the Euler characteric $\chi$, we obtain
$$\chi(\lk((a,b), A \times B)) = \chi(\lk(a,A)) + \chi(\lk(b,B)) - \chi(\lk(a,A)) \chi(\lk(b,B))$$
and we then deduce that
$$\partial(A \times B) = \partial A \times B \cup A \times \partial B.$$

%By additivity of the Euler characteristic $\chi$, we deduce
%$$\chi(\lk((a,b), A \times B)) = \chi(\{a\} \times \lk(b,B)) + \chi(\lk(a,A) \times \{b\}) + \chi(C'),$$ 
%where $C' := \{ \lambda(a, \beta) + (1-\lambda) (\alpha, b)~|~\lambda \in ]0,1[,~\beta \in \lk(b,B),~\alpha \in \lk(a,A) \}$.

%Notice that two segments $](a,\beta), (\alpha,b)[$ and $](a, \beta'), (\alpha',b)[$ of $C'$, with $(\alpha, \beta) \neq (\alpha', \beta')$, do not intersect, providing us a semialgebraic homeomorphism between $C'$ and the product $\left(\{a\} \times \lk(b,B)\right) \times ]0,1[ \times \left( \lk(a,A) \times \{b\} \right)$. 
%\\

%As a consequence we have
%$$\chi(\lk((a,b), A \times B)) = \chi(\lk(a,A)) + \chi(\lk(b,B)) - \chi(\lk(a,A)) \chi(\lk(b,B))$$
%and we then deduce that
%$$\partial(A \times B) = \partial A \times B \cup A \times \partial B.$$
%Therefore, if $(a,b) \in \partial(A \times B)$, that is by definition $\chi(\lk((a,b), A \times B)) \equiv 1 \mod 2$, we deduce from the above equality that $\chi(\lk(a,A)) \equiv 1 \mod 2$ or $\chi(\lk(b,B)) \equiv 1 \mod 2$ i.e. $a \in \partial A$ or $b \in \partial B$.

%Conversely, if $a \in \partial A$, i.e. $\chi(\lk(a,A)) \equiv 1 \mod 2$, then necessarily, $\chi(\lk((a,b), A \times B)) \equiv 1 \mod 2$. Consequently, $\partial A \times B \subset \partial(A \times B)$ and symmetrically $A \times \partial B \subset \partial(A \times B)$.

\end{proof}

\begin{lem} \label{lkprod} Let $(a,b) \in A \times B$, then the link $\lk((a,b), A \times B)$ of $(a,b)$ in $A \times B$ is semialgebraically homeomorphic to the set
$$\{ \lambda(a, \beta) + (1-\lambda) (\alpha, b)~|~\lambda \in [0,1],~\beta \in \lk(b,B),~\alpha \in \lk(a,A) \}.$$
\end{lem}

\begin{proof} Suppose $X \subset \mathbb{R}^n$ and $Y \subset \mathbb{R}^m$ for some $n,~m \gee 0$. 

Consider the continuous semialgebraic functions $p_1 : \mathbb{R}^n \times \{b\} \rightarrow \mathbb{R}~;~(x,b) \mapsto \|x - a\|^2$ and $p_2 : \{a\} \times \mathbb{R}^m \rightarrow \mathbb{R}~;~(a,y) \mapsto \|y - b\|^2$. By Hardt's theorem, for $\epsilon$ small enough, there exist semialgebraic trivializations over $]0, \epsilon^2]$
$$\overline{B}_n(a, \epsilon) \times \{b\} \rightarrow ]0, \epsilon^2] \times S_n(a, \epsilon) \times \{b\}~;~ (x,b) \mapsto (\|x - a\|^2, \widetilde{h}_1(x), b)$$
and
$$\{a\} \times \overline{B}_m(b, \epsilon) \rightarrow ]0, \epsilon^2] \times \{a\} \times S_m(b, \epsilon)~;~ (a,y) \mapsto (\|y - b\|^2,a, \widetilde{h}_2(y))$$
of $p_1$ and $p_2$ respectively, compatible with $A \times \{b\}$ and $\{a\} \times B$ respectively.
%, where $\widetilde{h}_1$ and $\widetilde{h}_2$ are continuous and semialgebraic such that $\widetilde{h}_{1 \,|S_n(a, \epsilon)} = Id$ and $\widetilde{h}_{2 \, |S_m(b, \epsilon)} = Id$.

%Symmetrically, there exists a semialgebraic trivialization of the function $p_2 : \{a\} \times \mathbb{R}^m \rightarrow \mathbb{R}~;~(a,y) \mapsto \|y - b\|^2$ over $]0, \epsilon^2]$, compatible with $\{a\} \times B$, given by
%$$\{a\} \times \overline{B}_m(b, \epsilon) \rightarrow ]0, \epsilon^2] \times \{a\} \times S_m(b, \epsilon)~;~ (a,y) \mapsto (\|y - b\|^2,a, \widetilde{h}_2(y))$$
%with $\widetilde{h}_2$ continuous and semialgebraic such that $\widetilde{h}_{2 \, |S_m(b, \epsilon)} = Id$.

We then define
$$f : S_{m+n}((a,b), \epsilon) \rightarrow C~;~(x,y) \mapsto \frac{\|x - a \|^2}{\epsilon^2}(\widetilde{h}_1(x),b) + \frac{\|y - b \|^2}{\epsilon^2}(a, \widetilde{h}_2(y)),$$ 
where $C := \{ \lambda(a, \beta) + (1-\lambda) (\alpha, b)~|~\lambda \in [0,1],~\beta \in  S_m(b,\epsilon),~\alpha \in S_n(a,\epsilon)\}$, which induces the desired semialgebraic homeomorphism.
%which is a semialgebraic homeomorphism. Since the trivializations of $p_1$ and $p_2$ over $]0,\epsilon^2]$ are compatible with $A \times \{b\}$ and $\{a\} \times B$ respectively, we have
%$$f(S_{n+m}((a,b), \epsilon) \cap A \times B) = \{ \lambda(a, \beta) + (1-\lambda) (\alpha, b)~|~\lambda \in [0,1],~\beta \in  S_m(b,\epsilon) \cap B,~\alpha \in S_n(a,\epsilon) \cap A \}.$$

\end{proof}

\subsection{Product and geometric filtration} \label{pgfsect}

Now we study the behaviour of the product of chains with respect to the geometric filtration :

\begin{prop} \label{prodchain}
$ \\$
(1) If $c \in \Gc_pC_q(X)$ and $c' \in \Gc_{p'}C_{q'}(Y)$, then
$$c \times c' \in \Gc_{p+p'}C_{q+q'}(X \times Y)$$
(2) If $c \in C_q(X)$, $c' \in C_{q'}(Y)$ and $c \times c' \in \Gc_{s}C_{q+q'}(X \times Y)$, then there exist $p,p' \in \mathbb{Z}$ with $p+p'=s$, such that
$$c \in \Gc_{p}C_q(X) \mathrm{\ and\ } c' \in \Gc_{p'}C_{q'}(Y)$$
\end{prop}

\begin{rem}
Because the filtration $\Gc_{\bullet}$ is increasing, the proposition shows in particular that the index $p+p'$ of the product $c \times c'$ in the filtration is minimal if and only if the indices $p$ of $c$ and $p'$ of $c'$ are minimal. In other words, if $c \in \Gc_pC_q(X) \setminus \Gc_{p-1}C_q(X)$ and $c' \in \Gc_{p'}C_{q'}(Y) \setminus \Gc_{p'-1}C_{q'}(Y)$, then
$$c \times c' \in \Gc_{p+p'}C_{q+q'}(X \times Y) \setminus \Gc_{p+p'-1}C_{q+q'}(X \times Y),$$
and if $c \in \Gc_pC_q(X)$ and $c' \in \Gc_{p'}C_{q'}(Y)$ with $c \times c' \notin \Gc_{p+p'-1}C_{q+q'}(X \times Y)$ then
$$c \notin \Gc_{p-1}C_q(X) \mbox{   and   } c' \notin \Gc_{p'-1}C_{q'}(Y).$$
\end{rem}

For the proof of proposition \ref{prodchain}, we use the notion of adapted resolutions (see \cite{MP}, section 2). Adapted resolutions allow us to work with chains lying in a nonsingular ambient space, with boundary belonging to a normal crossing divisor.

\begin{lem} \label{aresol}
Suppose $X$ is compact and consider a chain $c = [A]$ of $X$. We can assume that the dimension of $A$ is maximal, equal to the dimension of $X$, by considering the Zariski closure $\overline A^Z$ of $A$, since the filtration is only depending on the support of $c$ (Theorem {\bf 2.1} (1) of \cite {MP}) : we have
$$c \in \Gc_pC_k(X) \Longleftrightarrow c \in \Gc_pC_k(\overline A^Z).$$
With this assumption, there exists a resolution of singularities $\pi$ : $\widetilde X \longrightarrow X$ of $X$ such that $\supp ( \partial (\pi^{-1} c)) \subset D$, where $D$ is a normal crossing divisor of $\widetilde X$.
\end{lem}

Such a resolution is called a \emph{resolution of $X$ adapted to the chain $c$}. Notice that the pullback $\pi^{-1}c$ of $c$ is defined because $\pi$ is a resolution of singularities (the pullback operation on chains is more generally defined for any acyclic square of real algebraic varieties : see \cite{MP} Appendix).

\begin{proof}
First consider a resolution $\pi'$ : $X' \longrightarrow X$ of $X$ to make the ambient space nonsingular, then consider a resolution $\widetilde{X} \lra X'$ of the embedded variety $\supp \left( \partial (\pi'^{-1} c) \right)$ (which is a hypersurface of $X'$), so that it is in a normal crossing divisor.
\end{proof}

\begin{proof} {(\it of Proposition \ref{prodchain})}
The first point can be proved using the description of geometric filtration using Nash-constructible fonctions (\cite{MPACF}, \cite{MPACFj}) : see section 3 of \cite{MP}. There exist generically Nash-constructible functions $\varphi \ : \ X \longrightarrow 2^{q+p}\ZZ \mbox{  and  } \psi \ : \ Y \longrightarrow 2^{q'+p'}\ZZ$
such that
$$c = \left[\left\{ x \in X ~|~ \varphi(x) \notin 2^{q+p+1}\ZZ \right\}\right] \mbox{  and  } c' = \left[\left\{ y \in Y ~ | ~ \psi(y) \notin 2^{q'+p'+1}\ZZ \right\}\right].$$
Denote by $\pi_X$ : $X \times Y \longrightarrow X$ and $\pi_Y$ : $X \times Y \longrightarrow Y$ the projections. We define the function
$$\eta := \pi_X^*(\varphi) \cdot \pi_Y^*(\psi) \ :\ \begin{array}{rcl}
X \times Y & \longrightarrow & 2^{q+q'+p+p'}\ZZ \\
(x,y) & \longmapsto & \varphi(x)\cdot \psi(y)
\end{array}$$
which is Nash-constructible because the pullback of a Nash-constructible function and the product of Nash-constructible functions are Nash-constructible. Since,
$$c \times c' = \left[\left\{ (x,y) \in X \times Y ~|~ \eta(x,y) \notin 2^{q+q+p+p'+1}\ZZ \right\}\right],$$
the chain $c \times c'$ is in $\Gc_{p+p'}C_{q+q'}(X \times Y)$. 
\\

To prove the second point of the proposition, we use the very definition of the geometric filtration (see \cite{MP} Theorem 2.1 and Proposition 2.6). We first assume $X$ and $Y$ to be compact and we proceed by induction on the dimension of $X \times Y$. Suppose $c \times c' \in \Gc_{s}C_{q+q'}(X \times Y)$.

Let $\pi$ : $\widetilde X \longrightarrow X$ be an adapted resolution of $c$ (it exists by lemma \ref{aresol} above). Then, if we set $\widetilde {c} := \pi^{-1} (c)$, the support $\supp ( \partial \widetilde {c})$ is included in a normal crossing divisor $D$ of $\widetilde X$, and by definition of the geometric filtration, we have for $p \geq - q$
$$c \in \Gc_p C_q(X) \Longleftrightarrow \partial \widetilde {c} \in \Gc_p C_{q-1}(D).$$
In the same way, consider $\pi'$ : $\widetilde Y \longrightarrow Y$ an adapted resolution of $c'$. We have $\supp ( \partial \widetilde {c'}) \subset D'$, with $\widetilde {c'} := \pi'^{-1} (c')$ and $D'$ a normal crossing divisor of $\widetilde Y$, and for $p' \geq -q'$
$$c' \in \Gc_{p'} C_{q'}(Y) \Longleftrightarrow \partial \widetilde {c'} \in \Gc_{p'} C_{q'-1}(D').$$

Now $\pi \times \pi'$ : $\widetilde X \times \widetilde Y \longrightarrow X \times Y~;~(x,y) \mapsto \pi(x) \times \pi'(y)$ is an adapted resolution of $c \times c'$. Indeed, by Lemma \ref{bordprod},
\begin{equation} \label{boundsumres} 
\partial(\widetilde {c} \times \widetilde {c'}) = \partial \widetilde {c} \times \widetilde {c'} + \widetilde {c} \times \partial \widetilde {c'}, 
\end{equation}
in particular,
$$\supp (\partial(\widetilde {c} \times \widetilde {c'})) \subset (D \times \widetilde Y) \cup (\widetilde X \times D') :$$
the subvarieties $D \times \widetilde Y$ and $\widetilde X \times D'$ are normal crossing divisors of $\widetilde X \times \widetilde Y$ because $\widetilde X$ and $\widetilde Y$ are nonsingular, and their union that we denote by $\widetilde D$ is again a normal crossing divisor of $\widetilde X \times \widetilde Y$ since they have no common irreducible component (their intersection is $D \times D'$ which has strictly smaller dimension).
  
Therefore $c \times c' \in \Gc_{s}C_{q+q'}(X \times Y) \Leftrightarrow \partial(\widetilde {c} \times \widetilde {c'}) \in \Gc_s C_{q+q'-1} (\widetilde D)$. We then use the lemma \ref{lemdiv} below to deduce that $\partial(\widetilde {c}) \times \widetilde {c'} \in \Gc_s C_{q+q'-1} (D \times \widetilde Y)$ (and $\widetilde {c} \times \partial(\widetilde {c'}) \in \Gc_s C_{q+q'-1} (\widetilde X \times D')$).
 
By induction on dimension, there exists $p$ and $p'$ in $\mathbb{Z}$ such that $p + p' = s$, $\partial(\widetilde {c}) \in \Gc_p C_{q-1}(D)$ and $\widetilde {c'} \in \Gc_{p'} C_{q'}(\widetilde Y)$, that is $c \in \Gc_p C_{q}(X)$ and $c' \in \Gc_{p'} C_{q'}(Y)$ (since $\widetilde Y$ is nonsingular and $\widetilde {c'} \in D'$, we have $\widetilde {c'} \in \Gc_{p'} C_{q'}(\widetilde Y) \Leftrightarrow \partial(\widetilde {c'}) \in \Gc_p' C_{q'-1}(D')$).
\\

Finally, to prove the result in the general case, consider real algebraic compactifications $\overline{X}$ and $\overline{Y}$ of $X$ and $Y$ respectively. Then $\overline{X} \times \overline{Y}$ is a compactification of $X \times Y$ and, by definition of the geometric filtration,
$$c \times c' \in \Gc_{s}C_{q+q'}(X \times Y) \Leftrightarrow \overline{c \times c'} \in \Gc_{s}C_{q+q'}(\overline{X} \times \overline{Y})$$
and we can use the previous case ($\overline{c} \in \Gc_{p}C_q(\overline{X})$ if and only if $c \in \Gc_{p}C_q(X)$ and $\overline{c'} \in \Gc_{p'}C_{q'}(\overline{Y})$ if and only if $c' \in \Gc_{p'}C_{q'}(Y)$).

\end{proof}

\begin{lem} \label{lemdiv}
With the above notations, we have
$$\partial(\widetilde {c} \times \widetilde {c'}) \in \Gc_s C_{q+q'-1} (\widetilde{D}) \Longleftrightarrow \left\{\begin{array}{l}
\partial(\widetilde {c}) \times \widetilde {c'} \in \Gc_s C_{q+q'-1} (D \times \widetilde Y)\\
\mathrm{\ and\ }\\
\widetilde {c} \times \partial(\widetilde {c'}) \in \Gc_s C_{q+q'-1} (\widetilde X \times D')
\end{array}
\right.$$
\end{lem}

\begin{proof}
The implication from right to left follows from the definition of the geometric filtration ($D \times \widetilde Y$ and $\widetilde X \times D'$ are two subvarieties of $\widetilde D$) and the formula (\ref{boundsumres}).

We prove the implication from left to right using the description of the geometric filtration via Nash-constructible functions. 

First denote $A := \supp (\partial(\widetilde {c}) \times \widetilde {c'} )$, $B := \supp (\widetilde {c} \times \partial(\widetilde {c'}))$, then the closed semialgebraic set $A \cup B$ represents the chain $\partial(\widetilde {c} \times \widetilde {c'})$ in $C_{q + q'-1}(\widetilde D)$ (because $A \cap B \subset D \times D'$ is of strictly smaller dimension). Furthermore, because it belongs to $\Gc_s C_{q+q'-1} (\widetilde{D})$, the chain $\partial(\widetilde {c} \times \widetilde {c'})$ is represented by a Nash-constructible function $\varphi : \widetilde{D} \rightarrow 2^{q+q'-1 + s} \mathbb{Z}$ and we have
$$A \cup B = \{ (x,y) \in \widetilde D ~|~ \varphi(x,y) \notin 2^{q+q' + s} \mathbb{Z} \},$$
up to a set of dimension $< q + q' -1$. 

Consider now the characteristic functions $\psi_A$ and $\psi_B$ on $\widetilde{D}$ of the Zariski closures of $A$ and $B$ respectively. The Nash-constructible function $\varphi \cdot \psi_A : \widetilde{D} \rightarrow 2^{q+q'-1 + s} \mathbb{Z} $ represents the chain $\left[(A \cup B) \cap \overline{A}^{Z}\right]$ in $C_{q + q'-1}(\widetilde D)$. But, since the intersection of the Zariski closures of $A$ and $B$ is of dimension $< q + q'- 1$ (because it is a subvariety of $D \times D'$), we have $\left[(A \cup B) \cap \overline{A}^{Z}\right] = [A]$ and the Nash-constructible function $\varphi \cdot \psi_A : \widetilde{D} \rightarrow 2^{q+q'-1 + s} \mathbb{Z} $ represents the chain $[A] = \partial(\widetilde {c}) \times \widetilde {c'}$. Consequently, $\partial(\widetilde {c}) \times \widetilde {c'} \in \Gc_s C_{q+q'-1} (\widetilde D)$ and, since $\supp (\partial(\widetilde {c}) \times \widetilde {c'}) \subset D \times \widetilde Y$, we have
$$\partial(\widetilde {c}) \times \widetilde {c'} \in \Gc_s C_{q+q'-1} (D \times \widetilde Y).$$

Symmetrically, the Nash-constructible function $\varphi \cdot \psi_B : \widetilde{D} \rightarrow 2^{q+q'-1 + s} \mathbb{Z}$ represents ${c} \times \partial(\widetilde {c'})$ and  
$$\widetilde {c} \times \partial(\widetilde {c'}) \in \Gc_s C_{q+q'-1} (\widetilde X \times D').$$

\end{proof}

Next, we want to find a relation between the geometric filtration of the product variety $X \times Y$ and the product of the geometric filtrations of $X$ and $Y$. First, we have to make precise what we mean by a product of filtered complexes :

\begin{definition} \label{pfc}
Let $(K_{*},F)$ et $(M_{*},J)$ be two filtered complexes in the category $\Cc$. We define $\left( (K \otz M)_{*} , F \ot J \right)$ to be the complex given by
$$(K \otimes M)_n := \bigoplus_{i+j=n} K_i \otz M_j$$
equipped with the differential
$$d(x \otimes y) := dx \otimes y + x \otimes dy$$
and the bounded increasing filtration given by
$$(F \ot J)_p (K \ot M)_n := \bigoplus_{i+j=n} \sum_{a+b=p} F_a K_i \otz J_b M_j$$
\end{definition}

The latter filtration induces a spectral sequence converging to the homology of the product of the filtered complexes. In lemma \ref{prodcomp} below, we give the relation between this spectral sequence and some product of the spectral sequences induced by each filtered complex. This result will follow from the K\"unneth isomorphism (see for instance \cite{GR}, Theorem 29.10).

\begin{lem} \label{prodcomp} Let $(K_*, F)$ and $(M_*,J)$ be two filtered chain complexes. The spectral sequence of their product verifies, for $r \gee 0$, 
\begin{equation} \label{specprod}
E^r_{a,b} (K \otimes M) \cong \bigoplus_{p+ s = a, q + t = b} E^r_{p,q}(K) \otimes E^r_{s,t}(M).
\end{equation}
\end{lem}

\begin{proof} We prove this lemma by induction on $r$. First we have, for $a,~b \in \mathbb{Z}$, 
\begin{eqnarray*}
E^0_{a,b} (K \otimes M) & = & \frac{(F \otimes J)_a (K \otimes M)_{a+b}}{(F \otimes J)_{a-1} (K \otimes M)_{a+b}} \\
& = & \bigoplus_{i + j = a + b} \frac{\sum_{\alpha + \beta = a} F_{\alpha} K_i \otimes J_{\beta} M_j}{\sum_{\alpha + \beta = a-1} F_{\alpha} K_i \otimes J_{\beta} M_j}\\
& = & \bigoplus_{i+j = a+b} \bigoplus_{\alpha + \beta = a} \frac{ F_{\alpha} K_i}{F_{\alpha-1} K_i} \otimes \frac{ J_{\beta} M_j}{J_{\beta-1} M_j} \\
& = & \bigoplus_{p+ s = a, q + t = b} E^0_{p,q}(K) \otimes E^0_{s,t}(M) 
\end{eqnarray*}
The third equality is given by the lemma \ref{lemprodspec} below. 
\\
 
Then suppose the property is true for a fixed $r \gee 0$. The term $E^r(K \otimes M)$ of the spectral sequence induced by the filtered tensor product of $K_*$ and $M_*$ is composed of chain complexes $(E^r_{*,*}, d^r_{*,*})$ whose homology computes the term $E^{r+1}(K \otimes M)$. Applying homology and K\"unneth isomorphism to the formula (\ref{specprod}) given at level $r$ by induction, we obtain the property at level $r+1$. 

\end{proof}

\begin{lem} \label{lemprodspec} Let $a$ and $b$ be in $\mathbb{Z}$ and let $i,j \in \mathbb{Z}$ such that $i + j = a + b$. Keeping the notations from lemma \ref{prodcomp} above, we have 
$$\frac{\sum_{\alpha + \beta = a} F_{\alpha} K_i \otimes J_{\beta} M_j}{\sum_{\alpha + \beta = a-1} F_{\alpha} K_i \otimes J_{\beta} M_j} \cong \bigoplus_{\alpha + \beta = a} \frac{ F_{\alpha} K_i}{F_{\alpha-1} K_i} \otimes \frac{ J_{\beta} M_j}{J_{\beta-1} M_j}.$$
\end{lem}

\begin{proof}
Denoting simply $F_{\alpha} K_i$ by $F_{\alpha}$ and $J_{\beta} M_j$ by $J_{\beta}$, let $\psi$ be the $\mathbb{Z}_2$-linear map 
$$\sum_{\alpha + \beta = a} F_{\alpha} \otimes J_{\beta} \rightarrow \bigoplus_{\alpha + \beta = a} \frac{ F_{\alpha}}{F_{\alpha-1}} \otimes \frac{ J_{\beta}}{J_{\beta-1}}$$ 
(well-)defined by, if $x \otimes y \in F_{\alpha} \otimes J_{\beta}$, $\psi(x \otimes y) := \overline{x} \otimes \widehat{y} \in \frac{ F_{\alpha}}{F_{\alpha-1}} \otimes \frac{ J_{\beta}}{J_{\beta-1}}$. 
\\

The map $\psi$ is surjective and $\sum_{r+s = a -1} F_r \otimes J_s \subset \ker \psi$. Now let $\gamma \in \ker \psi$. Then, $\gamma = \sum_{\alpha + \beta = a} \sum_{i \in I_{\alpha, \beta}} x_i^{\alpha} \otimes y_i^{\beta}$ where, for all $\alpha, \beta \in \mathbb{Z}$ such that $\alpha + \beta = a$, $I_{\alpha, \beta}$ is finite and for all $i \in I_{\alpha, \beta}$, $x_i^{\alpha} \in F_{\alpha}$ and $y_i^{\beta} \in J_{\beta}$. We have
$$0 = \psi(\gamma) = \sum_{\alpha + \beta = a} \sum_{i \in I_{\alpha, \beta}} \overline{x_i^{\alpha}} \otimes \widehat{y_i^{\beta}},$$
Let $\alpha, \beta \in \mathbb{Z}$ such that $\alpha + \beta = a$, then $\sum_{i \in I_{\alpha, \beta}} \overline{x_i^{\alpha}} \otimes \widehat{y_i^{\beta}} = 0$, which means that, for all $i \in I_{\alpha, \beta}$, there exist $z_i^{\alpha-1} \in F_{\alpha-1}$ and $w_i^{\beta-1} \in J_{\beta-1}$ such that $\sum_{i} (x_i^{\alpha} + z_i^{\alpha-1}) \otimes (y_i^{\beta} + w_i^{\beta-1}) = 0$ i.e.
$$\sum_{i \in I_{\alpha, \beta}} x_i^{\alpha} \otimes y_i^{\beta} = \sum_{i \in I_{\alpha, \beta}} x_i^{\alpha} \otimes w_i^{\beta-1} + \sum_{i \in I_{\alpha, \beta}} z_i^{\alpha-1} \otimes y_i^{\beta} + \sum_{i \in I_{\alpha, \beta}} z_i^{\alpha-1} \otimes w_i^{\beta-1} \in F_{\alpha} \otimes J_{\beta-1} + F_{\alpha-1} \otimes J_{\beta}.$$
As a consequence, $\gamma \in \sum_{r+s = a -1} F_r \otimes J_s$ and $\ker \psi = \sum_{r+s = a -1} F_r \otimes J_s$. Since the map~$\psi$ induces an isomorphism $\sum_{\alpha + \beta = a} F_{\alpha} \otimes J_{\beta} / \ker \psi \rightarrow \im \psi$ by the universal property of the quotient, we get the result. 
\end{proof}

This lemma \ref{prodcomp} will allow us to relate the weight spectral sequence of a product of real algebraic varieties and the product of the weight spectral sequences, as an interpretation of the following result. We relate the product of geometric filtrations with the geometric filtration of the product : the map that associates to a tensor product of chains their cross product is a filtered quasi-isomorphism with respect to the geometric filtration.

\begin{thm} \label{prodhom}
We have a filtered quasi-isomorphism
$$u  : \Gc_{\bullet} C_*(X) \otimes \Gc_{\bullet} C_*(Y)  \lra  \Gc_{\bullet} C_*(X \times Y)$$
given by 
$$c_X \otimes c_Y  \longmapsto  c_X \times c_Y \in \Gc_{s} C_n(X \times Y)$$
if $c_X \in \mathcal{G}_p C_q(X)$ and $c_Y \in \mathcal{G}_{p'} C_{q'}(Y)$ with $p + p' = s$ and $q + q' = n$.
\end{thm}

The morphism $u$ is filtered by proposition \ref{prodchain}.

\begin{cor} \label{prodhomcor}
The filtered complexes $\Wcc(X) \otz \Wcc(Y)$ and $\Wcc(X \times Y)$ are isomorphic in $H o \, \mathcal{C}$ and the above map $u$ induces a filtered isomorphism
$$u_{\infty}  : \Wc_{\bullet} H_*(X) \otimes \Wc_{\bullet} H_*(Y) \lra \Wc_{\bullet} H_*(X \times Y).$$
\end{cor}

Above theorem \ref{prodhom} has an interpretation from the viewpoint of spectral sequences : $u$ is a morphism of filtered complexes which induces an isomorphism on spectral sequences from level one
$$u_r \ : \ E^r_{a,b} ( \Gc_{\bullet} C_*(X) \otimes \Gc_{\bullet} C_*(Y) ) \stackrel{\sim}{\lra} E^r_{a,b}(X \times Y)$$
for $r \gee 1$. Using Lemma \ref{prodcomp}, we get an isomorphism
\begin{equation} \label{hss}
u'_r \ : \ \bigoplus_{p+s=a, \ q+t=b} E^r_{p,q}(X) \otz E^r_{s,t}(Y) \stackrel{\sim}{\lra} E^r_{a,b}(X \times Y)
\end{equation}
(for $r \gee 1$). In particular, using this isomorphism and proposition \ref{vbnc}, we show the multiplicativity property of the virtual Poincaré polynomial
$$\beta( \, \cdot \, ) (u) := \sum_{q \gee 0} \beta_q(\, \cdot \,) u^q$$
without using the weak factorization theorem as in \cite{MPVB} and \cite{FMT}.

Furthermore, the Künneth isomorphism in homology is filtered through the isomorphisms
$$u'_{\infty} \ : \ \bigoplus_{p+s=a, q+t=b} E^{\infty}_{p,q}(X) \otz E^{\infty}_{s,t}(Y) \stackrel{\sim}{\lra} E^{\infty}_{a,b}(X \times Y)$$

\begin{rem}
To prove theorem \ref{prodhom}, we use the naturality property of the extension theorem of \cite{GNA} (Proposition 1.4 of \cite{MP}). We first show that $u$ is a filtered quasi-isomorphism for nonsingular projective real algebraic varieties and then use this naturality to prove that $u$ is a filtered quasi-isomorphism for all real algebraic varieties. We do not know whether the theorem follows directly from the geometric filtration of \cite{MP} and proposition \ref{prodchain}, without the theory of cubical hyperresolutions of \cite{GNA}.
\end{rem}

\begin{proof} {\it (of Theorem \ref{prodhom})} When $X$ and $Y$ are nonsingular projective varieties, so is the product variety $X \times Y$ and the three induced weight spectral sequences verify
$$E^{\infty}_{p,q} = E^1_{p,q} = \left\{ \begin{array}{ll}
 H_{p+q} &  \mathrm{\ if\ } p+q = -p\\
0 &  \mathrm{\ otherwise.}
\end{array} \right.$$
Therefore, by lemma \ref{prodcomp}, the morphism $u : \Gc_{\bullet} C_*(X) \otimes \Gc_{\bullet} C_*(Y)  \lra  \Gc_{\bullet} C_*(X \times Y)$ induces on $E^1 = E^{\infty}$ the morphism
$H_*(X) \otz H_*(Y) \lra H_*(X \times Y),$
which is the classical K\"unneth isomorphism in singular homology. Thus, $u$ is a filtered quasi-isomorphism when $X$ and $Y$ are projective nonsingular.
\\

Let $Y$ be now a fixed nonsingular projective variety and consider the two functors
$$\phi_1 : \mathbf{Sch}_c(\mathbb{R}) \rightarrow \mathcal{C}~;~ X \mapsto \mathcal{G} C(X) \otimes \mathcal{G} C(Y)$$
and
$$\phi_2 : \mathbf{Sch}_c(\mathbb{R}) \rightarrow \mathcal{C}~;~ X \mapsto \mathcal{G} C(X \times Y)$$
(in this part of the proof, we drop the subscripts of filtrations and complexes for readability). We proved above that these two functors are quasi-isomorphic in $\mathcal{C}$ on $\mathbf{V}(\mathbb{R})$ (we denote by $\varphi_1$ and $\varphi_2$ their respective restrictions to $\mathbf{V}(\mathbb{R})$). Furthermore, they both verify the additivity and acyclicity conditions of Theorem (2.2.2) of \cite{GNA}. Indeed, if $Z \hookrightarrow X$ is a closed inclusion, the additivity of the geometric filtration (see Theorem 2.7 of \cite{MP}) induces the exactness of the sequences
$$0 \rightarrow \mathcal{G} C(Z) \otimes \mathcal{G} C(Y) \rightarrow \mathcal{G} C(X) \otimes \mathcal{G} C(Y) \rightarrow \mathcal{G} C(X \setminus Z) \otimes \mathcal{G} C(Y) \rightarrow 0$$ 
and
$$0 \rightarrow \mathcal{G} C(Z \times Y) \rightarrow \mathcal{G} C(X \times Y) \rightarrow \mathcal{G} C((X \setminus Z) \times Y) \rightarrow 0$$
(the induced morphism $Z \times Y \hookrightarrow X \times Y$ is also a closed inclusion). Now, if the diagram
$$\begin{array}{ccc}
\widetilde Z & \rightarrow & \widetilde X\\
\downarrow & & \downarrow\\
Z & \rightarrow & X
\end{array}$$
is an acyclic square, we check that the simple filtered complexes associated to the induced diagrams
$$\begin{array}{ccc}
\mathcal{G} C(\widetilde Z) \otimes \mathcal{G} C(Y) & \longrightarrow & \mathcal{G} C(\widetilde X) \otimes \mathcal{G} C(Y)\\
\downarrow & & \downarrow\\
\mathcal{G} C(Z) \otimes \mathcal{G} C(Y) & \longrightarrow & \mathcal{G} C(X) \otimes \mathcal{G} C(Y)
\end{array}$$
denoted by $\Kc_1(Y)$, and
$$\begin{array}{ccc}
\mathcal{G} C(\widetilde Z \times Y) & \longrightarrow & \mathcal{G} C(\widetilde X \times Y)\\
\downarrow & & \downarrow\\
\mathcal{G} C(Z \times Y) & \longrightarrow & \mathcal{G} C(X \times Y)
\end{array}$$
denoted by $\Kc_2(Y)$, are acyclic. The simple filtered complex $\mathbf{s}\Kc_2(Y)$ is acyclic because the geometric filtration verifies the acyclicity condition for an acyclic square (see Theorem 2.7 of \cite{MP}) and the diagram 
$$\begin{array}{ccc}
\widetilde Z \times Y & \rightarrow & \widetilde X \times Y\\
\downarrow & & \downarrow\\
Z \times Y & \rightarrow & X \times Y
\end{array}$$
is acyclic. The spectral sequence induced by $\mathbf{s}\Kc_1(Y)$ verifies $E^1 = 0$ because $\mathbf{s}\Kc_1(Y)$ is nothing more than the tensor product of filtered complexes $\mathbf{s} \Kc_0 \otimes \mathcal{G} C(Y)$, where $\Kc_0$ is the diagram 
$$\begin{array}{ccc}
\mathcal{G} C(\widetilde Z) & \longrightarrow & \mathcal{G} C(\widetilde X)\\
\downarrow & & \downarrow\\
\mathcal{G} C(Z) & \longrightarrow & \mathcal{G} C(X)
\end{array}$$
and therefore, by lemma \ref{prodcomp}, $\displaystyle{E^1_{a,b}(\Kc_1(Y)) = \bigoplus_{p+s=a, q+t = b} E^1_{p,q}(\mathbf{s} \Kc_0) \otimes E^1_{s,t}(\mathcal{G} C(Y)) = 0}$, for all $a,b \in \mathbb{Z}$, again because of the acyclicity of the geometric filtration (notice that in both cases, we did not use the fact that $Y$ was projective nonsingular). 

Consequently, the localizations $\phi'_1, \phi'_2 : \mathbf{Sch}_c(\mathbb{R}) \rightarrow H o \, \mathcal{C}$ of $\phi_1$ and $\phi_2$ respectively are the unique extensions of their respective restrictions $\varphi_1', \varphi_2' : \mathbf{V}(\mathbb{R}) \rightarrow H o \, \mathcal{C}$ given by the Theorem 2.2.2 of \cite{GNA} (notice that the above arguments also prove that the functors $\varphi_1$ and $\varphi_2$ satisfy the disjoint additivity condition $(F_1)$ and the elementary acyclicity condition $(F_2)$). By the naturality of this extension (see Proposition 1.4 of \cite{MP}), the localization of the filtered quasi-isomorphism $u(Y) : \varphi_1 \rightarrow \varphi_2$ extends uniquely into a morphism $\phi'_1 \rightarrow \phi'_2$, and this morphism is an isomorphism of $H o \, \mathcal{C}$. Since the localization of 
$$u(Y) : \phi_1 \rightarrow \phi_2~;~X \mapsto (\Gc_{\bullet} C_*(X) \otimes \Gc_{\bullet} C_*(Y)  \lra  \Gc_{\bullet} C_*(X \times Y))$$
is such an extension, the latter is a quasi-isomorphism of $\mathcal{C}$, that is the morphism $u : \Gc_{\bullet} C_*(X) \otimes \Gc_{\bullet} C_*(Y)  \lra  \Gc_{\bullet} C_*(X \times Y)$ is a filtered quasi-isomorphism for any real algebraic variety $X$ and $Y$ a nonsingular projective variety.
\\

Now fix $X$ to be any real algebraic variety. We prove in the same way as above that
$$u(X) : Y \in \mathbf{Sch}_c(\mathbb{R}) \mapsto (\Gc_{\bullet} C_*(X) \otimes \Gc_{\bullet} C_*(Y)  \lra  \Gc_{\bullet} C_*(X \times Y))$$
is a quasi-isomorphism of $\mathcal{C}$. As a consequence,
$$u : \Gc_{\bullet} C_*(X) \otimes \Gc_{\bullet} C_*(Y)  \lra  \Gc_{\bullet} C_*(X \times Y)$$
is a filtered quasi-isomorphism for any real algebraic varieties $X$ and $Y$.
\end{proof}

\begin{rem} \label{prodextcan}
A morphism between the filtered complexes $\Wcc(X) \ot \Wcc(Y)$ and $\Wcc(X \times Y)$ for any varieties $X$ and $Y$ can also be obtained without using the geometric filtration. Indeed, using a method similar to the one in the previous proof, we can extend to all real algebraic varieties the morphism of filtered complexes
$$F^{can}_{\bullet} C_*(X) \ot F^{can}_{\bullet} C_*(Y) \lra F^{can}_{\bullet} C_*(X \times Y)$$
(given by the product in definition \ref{defprod}) restricted to nonsingular projective varieties.
\end{rem}

\subsection{Product and cohomological weight complex} \label{pcwcsect}

As for the homological weight complex, we show that we can relate the cohomological weight complex of a product with the tensor product of the cohomological weight complexes, so that these two filtered complexes are isomorphic in the localized category $H o \, \mathfrak{C}$. More precisely, this isomorphism of $H o \, \mathfrak{C}$ is induced by two opposite-directional filtered quasi-isomorphisms of cochain complexes, one of them being the dualization of the quasi-isomorphism of $\mathcal{C}$ in theorem~\ref{prodhom}.

\begin{prop} \label{prodcohom}
The filtered complexes
$\Gc^{\bullet} C^*(X) \otimes \Gc^{\bullet} C^*(Y)$ and $\Gc^{\bullet} C^*(X \times Y)$ are isomorphic in $H o \, \mathfrak{C}$.\\
\end{prop}

\begin{cor} \label{prodcohomcor}
The filtered complexes $\WC(X) \ot \WC(Y)$ and $\WC(X \times Y)$ are isomorphic in $H o \, \mathfrak{C}$ and the K\"unneth isomorphism in cohomology
$$\Wc^{\bullet} H^*(X) \otimes \Wc^{\bullet} H^*(Y) \lra \Wc^{\bullet} H^*(X \times Y)$$
is a filtered isomorphism with respect to the cohomological weight filtration.\\
\end{cor}

\begin{proof} {\it (of Proposition \ref{prodcohom})}

Consider the filtered quasi-isomorphism 
$$u  : \Gc_{\bullet} C_*(X) \otimes \Gc_{\bullet} C_*(Y)  \lra  \Gc_{\bullet} C_*(X \times Y)$$
defined in theorem \ref{prodhom}. Its dual 
$$u^{\vee} \ : \ \begin{array}{rcl} (C_*(X \times Y))^{\vee} & \lra & (C_*(X) \ot C_*(Y) )^{\vee}\\
\eta & \lmt & \dis \left[ \sum_i c_{X,i} \otimes c_{Y,i} \lmt \sum_i \eta( c_{X,i} \times c_{Y,i}) \right]
\end{array}$$
is also a filtered quasi-isomorphism if we equip the dualized complexes with the corresponding dual filtrations (remark \ref{grad}).

On the other hand, the map
$$w \ : \  \begin{array}{ccl}
( C_*(X))^{\vee} \ot ( C_*(Y))^{\vee} & \lra & ( C_*(X) \ot C_*(Y) )^{\vee}\\
\varphi \otimes \psi & \lmt & \dis \left[ \sum_i c_{X,i} \otimes c_{Y,i} \lmt \sum_i \varphi(c_{X,i}) \cdot \psi(c_{Y,i}) \right]
\end{array}$$
where the right-hand side complex is equipped with the same filtration as above (induced by the geometric filtrations of $X$ and $Y$) and the left-hand side complex is the filtered tensor product of the dual geometric filtrations of $X$ and $Y$, is also a filtered quasi-isomorphism. Indeed, for $a,b \in \mathbb{Z}$, we have
$$E_1^{a,b} \left(( C_*(X))^{\vee} \otz ( C_*(Y))^{\vee}\right) = \bigoplus_{p+s=a, q+t=b} \left(E^1_{p,q}(X)\right)^{\vee} \otz \left(E^1_{s,t}(Y)\right)^{\vee},$$
and 
$$E_1^{a,b} \left( ( C_*(X) \otz C_*(Y) )^{\vee} \right) = \left(E^1_{a,b} \left(C_*(X) \otz C_*(Y) \right) \right)^{\vee} = \bigoplus_{p+s=a, q+t=b} \left(E^1_{p,q}(X) \otz E^1_{s,t}(Y)\right)^{\vee},$$
by (the cohomological version of) lemma \ref{prodcomp}, and the morphism $w$ induces on the $E_1$-level the morphisms
$$\left(E^1_{p,q}(X)\right)^{\vee} \otz \left(E^1_{s,t}(Y)\right)^{\vee} \longrightarrow \left(E^1_{p,q}(X) \otz E^1_{s,t}(Y)\right)^{\vee},$$
given by $\overline{\varphi} \otimes \widehat{\psi} \mapsto \left[ \, \sum_i \overline{c_{X,i}} \otimes \widehat{c_{Y,i}} \mapsto \sum_i \varphi(c_{X,i}) \cdot \psi(c_{Y,i}) \, \right]$, which are isomorphisms (the terms of the weight spectral sequences of $X$ and $Y$ are finite-dimensional).
\\

Therefore, we have the following diagram in $\mathfrak{C}$ 
$$(C_*(X \times Y))^{\vee}  \xrightarrow{u^{\vee}} (C_*(X) \otz C_*(Y) )^{\vee} \xleftarrow{w} ( C_*(X))^{\vee} \otz ( C_*(Y))^{\vee},$$
where the morphisms $u^{\vee}$ and $w$ are filtered quasi-isomorphisms. Consequently, in the localization $H o \, \mathfrak{C}$ of $\mathfrak{C}$ with respect to filtered quasi-isomorphisms, the filtered complexes $\Gc^{\bullet} C^*(X) \otimes \Gc^{\bullet} C^*(Y)$ and $\Gc^{\bullet} C^*(X \times Y)$ are isomorphic. 

\end{proof}

\begin{rem} \label{cancohom}
As for the homological case, a morphism between the filtered complexes $\WC(X) \ot \WC(Y)$ and $\WC(X \times Y)$ can be obtained without using the geometric filtration. Indeed, in the previous proof, consider the canonical filtration in place of the geometric filtration : one can show in the same way that there is an isomorphism of $H o \, \mathfrak{C}$ between $(F^{can})^{\bullet}_{\vee} C^*(X) \ot (F^{can})^{\bullet}_{\vee} C^*(Y)$ and $(F^{can})^{\bullet}_{\vee} C^*(X \times Y)$. Since the dual canonical filtration and the cohomological canonical filtration are filtered quasi-isomorphic (see the proof of theorem \ref{coincid}), we deduce an isomorphism of $H o \, \mathfrak{C}$ between $F_{can}^{\bullet} C^*(X) \ot F_{can}^{\bullet} C^*(Y)$ and $F_{can}^{\bullet} C^*(X \times Y)$. Restricting this isomorphism to projective nonsingular varieties, we extend it to all real algebraic varieties in the same way as in proof of theorem \ref{prodhom} (see also remark \ref{prodextcan}) to obtain an isomorphism between $\WC(X) \ot \WC(Y)$ and $\WC(X \times Y)$.
\end{rem}

\subsection{Cup product} \label{cupsect}

Let $X$ be a real algebraic variety. 

We show below that the cup product on the cohomology with compact supports $H^*(X)$ of the set of real points of $X$ is filtered with respect to the cohomological weight filtration. Precisely, we define a cup product on the cochain level in the derived category $H o \, \mathfrak{C}$, on the cohomological geometric filtration, using the filtered quasi-isomorphisms $w$ and $u^{\vee}$ defined above in the proof of \ref{prodcohom}, that induces a cup product on the cohomological weight spectral sequence of $X$ and the usual cup product on the cohomology of $X$. 
\\

Let $\Delta$ denote the diagonal map
$$\Delta \ : \ \begin{array}{rcl}
X & \longrightarrow & X \times X\\
x & \longmapsto & (x,x)
\end{array}$$
Now consider the cohomological geometric filtration $\mathcal{G}^{\bullet} C^*(X)$ of $X$ as an object of the derived category $H o \, \mathfrak{C}$. We can apply the composition $\Delta^* \circ (u^{\vee})^{-1} \circ w$ to the tensor product $\mathcal{G}^{\bullet} C^*(X) \otimes \mathcal{G}^{\bullet} C^*(X)$ ($u^{\vee}$ is an isomorphism of $H o \, \mathfrak{C}$, see the proof of proposition \ref{prodcohom}) :
$$\Delta^* \circ (u^{\vee})^{-1} \circ w : \mathcal{G}^{\bullet} C^*(X) \otimes \mathcal{G}^{\bullet} C^*(X) \xrightarrow{(u^{\vee})^{-1} \circ w} \mathcal{G}^{\bullet} C^*(X \times X) \xrightarrow{\Delta^*} \mathcal{G}^{\bullet} C^*(X)$$
We denote this morphism of $H o \, \mathfrak{C}$ by $\smile$.

\begin{prop}  \label{cupprodinf} $\\ $
The cup product
$$\smile~: \mathcal{G}^{\bullet} C^*(X) \otimes \mathcal{G}^{\bullet} C^*(X) \longrightarrow \mathcal{G}^{\bullet} C^*(X)$$
in $H o \, \mathfrak{C}$ induces a morphism of spectral sequences
$$\smile'_r \ : \ \bigoplus_{p+s=a, \ q+t=b} E_r^{p,q}(X) \otz E_r^{s,t}(X) \lra E_r^{a,b}(X)$$
and the usual cup product
$$\begin{array}{ccl}
H^*(X) \otz H^*(X) & \stackrel{\smile}{\longrightarrow} & H^*(X)\\
\varphi \otimes \psi & \longmapsto & \varphi \smile \psi =  [ \Delta^*(\varphi \times \psi) ]
\end{array}.$$
In particular, the cup product in cohomology is a filtered map with respect to the cohomological weight filtration.
\end{prop}

\begin{proof}
The first fact follows from the cohomological version of lemma \ref{prodcomp}, and the cup product in cohomology is the composition of $\Delta^*$ and the K\"unneth isomorphism in cohomology, which is itself induced by $(u^{\vee})^{-1} \circ w$ (see proposition \ref{prodcohom} and corollary \ref{prodcohomcor}).
\end{proof}

\subsection{Cap product} \label{capsect}

In this section, we define a cap product on the homological and cohomological geometric filtrations considered in the corresponding derived categories $H o \, \mathcal{C}$ and $H o \, \mathfrak{C}$. This cap product on chain level induces a cap product on the homological and cohomological weight spectral sequences, showing that the cap product on homology and cohomology is a filtered morphism with respect to the homological and cohomological weight filtrations.
\\

First, we give a filtered chain complex structure to the tensor product of a filtered cochain complex and a filtered chain complex :

\begin{definition} \label{pfc2}
Let $(K^{*},F)$ and $(M_{*},J)$ be respectively a filtered cochain complex of $\mathfrak{C}$ and a filtered chain complex of $\Cc$. We define $\left( (K^* \otimes M_* )_{*} , F \ot J \right)$ to be the chain complex given by
$$(K \otimes M)_n := \bigoplus_{j-i=n} K^i \otz M_j,$$
equipped with the differential
$$\partial(x \otimes y) := dx \otimes y + x \otimes \partial y$$
and the bounded increasing filtration given by
$$(F \ot J)_p (K \otimes M)_n = \bigoplus_{j-i=n} \sum_{b-a=p} F^a K^i \otimes J_b M_j$$
\end{definition}

Considering the semialgebraic chain and cochain complexes $C_*(X)$ and $C^*(X)$ of $X$, implicitely equipped with the homological and cohomological geometric filtrations (for sake of readability), as objects of the respective derived categories $H o \, \mathcal{C}$ and $H o \, \mathfrak{C}$, we are going to define a cap product $C^*(X) \otimes C_*(X) \rightarrow C_*(X)$ in $H o \, \mathcal{C}$. 
\\

First, let $\omega$ denote the morphism $C^*(X) = (C_*(X))^{\vee} \rightarrow (C^*(X) \otimes C_*(X))^{\vee}$ of $H o \, \mathfrak{C}$ given by
$$\begin{array}{rcl} (C_{l-m}(X))^{\vee} & \lra & (C^m(X) \otimes C_{l}(X) )^{\vee}\\
\psi & \lmt & \dis \left[ \varphi \otimes c \lmt (\psi \smile \varphi)(c) \right]
\end{array}$$

The cap product that we define below will be obtained from the dual of this filtered morphism, in order to have a formula 
\begin{equation} \label{eqcupcap} \psi(\varphi \frown c) = (\psi \smile \varphi)(c) \end{equation} on the chain level. We make precise what we mean by the dual filtered chain complex of a filtered cochain complex :

\begin{definition} If $F^{\bullet} K^*$ is a filtered cochain complex of $\mathfrak{C}$, we define its dual filtered chain complex $F^{\vee}_{\bullet} K^{\vee}_*$ of $\mathcal{C}$ by
$$F^{\vee}_p K^{\vee}_q := \left\{\eta \in K^{\vee}_q~|~ \eta \equiv 0 \mbox{  on  } F^{p+1} K^q \right\}.$$
\end{definition}

Notice that, as in remark \ref{grad}, we have the natural isomorphism of spectral sequences given by $E^r_{a,b}(F^{\vee} K^{\vee}) = \left(E_r^{a,b}(F K)\right)^{\vee}$. 
\\

Consider the dual filtered chain complexes $(C_*(X))^{\vee \vee}$ and $(C^*(X) \otimes C_*(X))^{\vee \vee}$ of $(C_*(X))^{\vee}$ and $(C^*(X) \otimes C_*(X))^{\vee}$ respectively. We have natural filtered morphisms $\nu : C_*(X) \rightarrow (C_*(X))^{\vee \vee}$ and $\mu : C^*(X) \otimes C_*(X) \rightarrow (C^*(X) \otimes C_*(X))^{\vee \vee}$, inducing the natural morphims $E^r_{a,b} \rightarrow \left(E^r_{a,b}\right)^{\vee \vee}$ on the spectral sequence level, which are isomorphisms from $r \gee 1$ (the terms of the spectral sequence are finite-dimensional from level one). 

Therefore, the morphims $\nu$ and $\mu$ are quasi-isomorphisms of $\mathcal{C}$ and we can define the morphism 
$$\nu^{-1} \circ \omega^{\vee} \circ \mu : C^*(X) \otimes C_*(X) \rightarrow C_*(X)$$
of $H o \, \mathcal{C}$ given by 
$$\begin{array}{rcl} C^m(X) \otimes C_{l}(X) & \lra & C_{l-m}(X)\\
\varphi \otimes c & \lmt & \varphi \frown c := \nu^{-1} \circ \omega^{\vee} \circ \mu(\varphi \otimes c) 
\end{array}$$ 

We denote it also by $\frown$ and we have : 

\begin{prop}
The cap product on the geometric filtrations of $X$ induces a cap product 
$$\ E_{r}^{p,q}(X) \otimes E^{r}_{s,t}(X) \lra E^{r}_{s-p, \ t-q}(X)$$
on the weight spectral sequences of $X$, and the usual cap product 
$$H^*(X) \otimes H_*(X) \stackrel{\frown}{\longrightarrow} H_*(X)$$
on the homology and cohomology of $X$. In particular, the latter is a filtered morphism with respect to the weight filtrations (the filtration on the tensor product of cohomology and homology is defined in a way similar to definition \ref{pfc2}). 
\end{prop}

\begin{proof} Similarly to lemma \ref{prodcomp}, the term of level $r$ and indices $a,b$ of the spectral sequence induced by $C^*(X) \otimes C_*(X)$ is given by $\displaystyle{\bigoplus_{s-p=a, \ t-q=b} E_r^{p,q}(X) \otimes E^r_{s,t}(X)}$. Then the cap product on chains and cochains induces morphisms
$$\frown'_r : \bigoplus_{s-p=a, \ t-q=b} E_r^{p,q}(X) \otimes E^r_{s,t}(X) \lra E_r^{a,b}(X).$$
Now, notice that the formula (\ref{eqcupcap}) on the chain level induce that, if $\varphi \in  E_{r}^{p,q}(X)$, $c \in E^{r}_{s,t}(X)$ and $\psi \in E_{r}^{p-s,q-t}(X)$ (then $\psi \smile \varphi \in E_{r}^{p,q}(X)$, which is isomorphic to $\left(  E^{r}_{p,q}(X) \right)^{\vee}$), we have
$$\psi(\varphi \frown'_r c) = (\psi \smile'_r \varphi)(c).$$

Since the cup product on the cohomological weight spectral sequence induces the cup product on cohomology and because the cap product on cohomology and homology
$$H^m(X) \otimes_{\mathbb{Z}_2} H_{l}(X) \stackrel{\frown}{\longrightarrow} H_{l-m}(X)$$
is characterized by the formula
$$\psi(\varphi \frown c) = (\psi \smile \varphi)(c)$$
(if $\varphi \in H^m(X)$ and $c \in H_l(X)$, $\varphi \frown c$ is the unique element of $H_{l-m}(X)$ verifying this formula for all $\psi \in H^{l-m}(X)$), the cap product on the cohomological and homological weight spectral sequences induces the cap product on cohomology and homology.
\end{proof}

\begin{rem}
\begin{enumerate}
	\item If $\varphi \in  E_{r}^{p,q}(X)$ and $c \in E^{r}_{s,t}(X)$ then $\varphi \frown'_r c$ is the unique element of $E^{r}_{s-p, \ t-q}(X)$ verifying
$$\psi(\varphi \frown'_r c) = (\psi \smile'_r \varphi)(c)$$
for all $\psi \in E_{r}^{p-s,q-t}(X)$.
	\item Another possible definition for the cap product on the chain level is the following one (see \cite{Spa}). Consider the morphism
	$$h : \begin{array}{rcl} C^*(X) \otimes (C_*(X) \otimes C_*(X)) & \lra & C_*(X) \\
\varphi \otimes (a \otimes b) & \lmt & \varphi(a) \cdot b
\end{array}.$$
Then we can also define the cap product on the cohomological and homological geometric filtrations of $X$ (regarded as objects of $H o \, \mathfrak{C}$ and $H o \, \mathcal{C}$) by setting
$$ \varphi \frown c := h ( \varphi \otimes u^{-1}(\Delta_*(c))).$$
Notice that this definition would be valid with integer coefficients as well.
\end{enumerate}
\end{rem}

\subsection{Weight filtrations and Poincar\'e duality map} \label{wfpdmsect}

Let $X$ be a compact real algebraic variety of dimension $n$. 

The semialgebraic chain $[X]$ is pure, that is $[X] \in \Gc_{-n} C_n(X)$. For $r \gee 1$, it induces homology classes in the weight spectral sequence terms $E^{r}_{-n,2n}(X)$. \\ 

By taking the cap product with $[X]$, we obtain a map $D$ on the cohomological weight spectral sequence of $X$, given by :
\begin{equation*}
D_r^{s,t} := \cdot \frown [X] \ : \ \begin{array}{ccl}
E_{r}^{s,t}(X) & \lra & E^{r}_{-n-s,2n-t}(X)\\
\varphi & \longmapsto & \varphi \frown [X]
\end{array}
\end{equation*}

Recall that the non-zero terms of the weight spectral sequences lie in the triangle given by the inequalities $t \gee - 2s$, $s \lee 0$ and $t \lee -s+n$, the terms induced by the pure chains lying in the line $t = - 2s$. Then if, for any $r \gee 1$, we consider the cap product of non-pure classes by $[X]$, it is identically zero.
\\

The map $D$ on the cohomological weight spectral sequence induces, on the $E^{\infty}$ and $E_{\infty}$ level, the classical Poincar\'e duality map on the cohomology of $X$ (that we denote again by $D$) given by
$$\ \begin{array}{ccl}
H^k(X) & \lra & H_{n-k}(X)\\
\varphi & \longmapsto & \varphi \frown [X]
\end{array}$$
($[X]$ corresponds here to the fundamental homology class of $X$) and :

\begin{prop} For all $p$ and $k$ in $\mathbb{Z}$, the image of $\mathcal{W}^p H^k(X)$ by Poincar\'e duality map is in $\mathcal{W}_{-p-n} H_{n-k}(X)$ : 
$$D(\mathcal{W}^p H^k(X)) \subset \mathcal{W}_{-p-n} H_{n-k}(X).$$

In particular, for all $k \in \mathbb{Z}$, $D(H^k(X)) \subset \mathcal{W}_{k-n} H_{n-k}(X)$ and, if $p > -k$, $D(\mathcal{W}^p H^k(X)) =~0$. In other words, all the non-pure cohomology classes are in the kernel of Poincar\'e duality map and the pure cohomology classes are the only classes which may be sent to a nonzero pure homology class by Poincar\'e duality map. Therefore, if its weight filtrations are not pure, a real algebraic variety does not satisfy Poincar\'e duality.
\end{prop}

\begin{rem} \label{ex1}
On the other hand, there exist varieties having pure weight filtration but not satisfying Poincaré duality. 

For example, let $X$ denote the pinched torus, obtained from a torus $T$ by identifying a circle which generates it as a revolution surface to a point. To compute its weight spectral sequence, we consider the cubical hyperresolution of $X$ given by the blowing-up at $x_0$ :
$$\begin{array}{ccc}
S^1 & \hra & T\\
\downarrow & & \downarrow\\
\bullet & \hra & X
\end{array}$$
We obtain a pure weight filtration given by the term $\widetilde E^{2} = \widetilde E^{\infty}$ :
$$\widetilde E^{2} = \left\lfloor  \begin{array}{ccc}
\ZZ_2 \cdot [X]&   & \\
\ZZ_2 \cdot \overline{[a]}& 0 & \\
\ZZ_2 & 0 & 0
\end{array} \right.$$
(if $H_1(T) = \mathbb{Z}_2 \overline{[a]} \oplus \mathbb{Z}_2 \overline{[b]}$ with $b = S^1$ the exceptional divisor of the blowing-up). However, the variety $X$ does not satisfy Poincaré duality since $\overline{[a]}^{\vee} \frown [X] = 0$.
\end{rem}

\vspace{0.5cm}
Thierry LIMOGES
\\
Lyc\'ee Lamartine
\\
381 avenue des Gaises
\\
71000 M\^acon
\\
FRANCE
\\
thierry.limoges@laposte.net
\vspace{0.5cm}
\\
Fabien PRIZIAC
\\
Department of Mathematics
\\
Faculty of Science
\\
Saitama University
\\
255 Shimo-Okubo, Sakura-ku, Saitama City
\\
338-8570
\\
JAPAN
\\
priziac.fabien@gmail.com

\end{document}